\documentclass[11pt,reqno,a4paper]{amsart}
\usepackage{amssymb,latexsym,amsmath,amsthm,pstricks}

\newcommand{\N}{\mathbb N}
\newcommand{\Z}{\mathbb Z}

\newcommand{\R}{\mathbb R}

\newcommand{\D}{\mathbb D}
\newcommand{\T}{\mathbb T}

\newcommand{\NN}{\mathbb N}
\newcommand{\ZZ}{\mathbb Z}

\newcommand{\RR}{\mathbb R}
\newcommand{\CC}{\mathbb C}
\newcommand{\DD}{\mathbb D}
\newcommand{\TT}{\mathbb T}

\renewcommand{\AA}{\mathbb A}
\renewcommand{\SS}{\mathbb S}

\newcommand{\cM}{\mathcal M}

\newtheorem{thm}{Theorem}

\newtheorem{corol}{Corollary}[section]
\newtheorem{lem}[corol]{Lemma}
\newtheorem{pro}[corol]{Proposition}
\newtheorem{facts}[corol]{Facts}

\theoremstyle{remark}
\newtheorem{claim}{Claim}
\newtheorem{rem}[corol]{Remark}

\theoremstyle{definition}
\newtheorem{defi}[corol]{Definition}

\begin{document}

\setlength{\leftmargini}{18pt}

\title[Existence of orbits with non-zero torsion]{Existence of orbits with non-zero torsion for certain types of surface diffeomorphisms}
\author{ F.\,B\'eguin and Z.\,Rezig Boubaker}
\address{F.\,B\'eguin, Universit\'{e} Paris-Sud 11, D\'{e}partement de Math\'{e}matiques, 91405
Orsay, France.}
 \address{Z.\,Rezig Boubaker, UniversitŽ du 7 Novembre \`a Carthage, Facult\'{e} des Sciences de Bizerte, D\'{e}partement de Math\'{e}matiques, 7021 Zarzouna, Tunisie.}
\date{\today{}}

\begin{abstract}
The present paper concerns the dynamics of surface diffeomorphisms. Given a diffeomorphism $f$ of a surface $S$, the \emph{torsion} of the orbit of a point $z\in S$ is, roughly speaking, the average speed of rotation of the tangent vectors under the action of the derivative of $f$, along the orbit of $z$ under $f$. The purpose of the paper is to identify some situations where there exist measures and orbits with non-zero torsion. We prove that every area preserving diffeomorphism of the disc which coincides with the identity near the boundary has  an orbit with non-zero torsion. We also prove that a diffeomorphism of the torus $\TT^2$, isotopic to the identity, whose rotation set has non-empty interior, has an orbit with non-zero torsion.
\end{abstract}


\maketitle

\section{Introduction}

Numerical conjugacy invariants are a key tool to analyze the behavior of dynamical systems. The paradigmatic example is of course Poincar\'e's rotation number for circle homeomorphisms~: a single numerical conjugacy invariant, the rotation number, allows to describe completely the dynamics of a circle diffeomorphism, at least when this diffeomorphism is smooth enough and the rotation number is  irrational. Unfortunately, this situation is quite specific to the circle. Consider for example a homeomorphism $f$ of the torus $\TT^2$, which is isotopic to the identity. Poincar\'e's construction of rotation numbers may be generalized, yielding rotation vectors for $f$. Nevertheless, unlike what happens for circle  homeomorphisms, different points of $\TT^2$ may have different rotation vectors. Moreover, even the collection of all the rotation vectors of the points in $\TT^2$ is far from describing completely  the dynamics of $f$. This is the reason why many other invariants have been defined to study the dynamics of torus (or, more generally, surfaces) homeomorphisms. One may for example consider the average speed at which two given orbits turn around each other. Or, for 
a diffeomorphism $f$ of $\SS^{2} \simeq \bar{\CC}$, one may construct a conjugacy invariant by considering the evolution of the cross ratio of four points under the action of $f$ (see \emph{e.g.} \cite{GG:95} for more details and many more examples).

In the present paper, we consider a numerical invariant which measures the average rotation speed of tangent vectors under the action of the derivative of a surface diffeomorphism. Consider a (non necessarily  compact) surface $S$ with trivializable tangent bundle and a diffeomorphism $f$ of $S$ isotopic to the identity. Choose an isotopy $I=(f_t)_{t\in [0,1]}$ joining the identity to $f$. For $t \in \RR,$ we define $f_t=f_{t-n} \circ f^{n}$ where $n=\lfloor t\rfloor$. Choose a trivialization of the tangent bundle of $S$. Then, for every point $x \in S$, every vector $\xi \in T_xS \setminus\{0\}$ and every $t$, we can see $df_t(x).\xi$ as a non-zero vector in $\RR^2 \simeq \CC$.  We denote by $\mathrm{Torsion}_n(I,x,\xi)$ the variation of the argument of $df_t(x).\xi$ when $t$ runs from $0$ to $n$, divided by $n$ (as a unit for angles, we use the full turn instead of the radian). If the limit $\lim_{n\rightarrow +\infty} \mathrm{Torsion}_n (I,x,\xi) $ exists, then it does not depend on $\xi$; we call it \emph{torsion of the orbit of $x$}, and denote it by $\mathrm{Torsion}(I,x)$. If $\mu$ is an $f$-invariant probability measure, then $\mathrm{Torsion}(I,x)$ exists for $\mu$-almost every $x$ and the function $x \mapsto \mathrm{Torsion}(I,x)$ is $\mu-$integrable; we call \emph{torsion of the measure $\mu$} the integral $\int_{S} \mathrm{Torsion}(I,x) d\mu (x)$ and denote it by $\mathrm{Torsion}(I,\mu)$. If $f$ has compact support in $S$ and if we only consider isotopies with compact support, then the quantities $\mathrm{Torsion}(I,x)$ and $\mathrm{Torsion}(I,\mu)$ do not depend on the choice of the isotopy $I$. Similarly, if $S$ is the torus $\T^{2}=\RR^2 / \ZZ^2$ or the annulus $\AA= \RR/\ZZ \times  \RR$ and if we use the canonical trivialization of the tangent bundle of $S$, then the quantities $\mathrm{Torsion}(I,x)$ and $\mathrm{Torsion}(I,\mu)$ do not depend on the choice of the isotopy~$I$. In those two cases, we may speak of the torsion of an orbit of $f$ without having chosen an isotopy (see section~\ref{s.definition-torsion} for more details).

Similar notions have been considered by several authors. What we call the \emph{torsion of the area probability measure} has first been defined by D.Ruelle for conservative diffeomorphisms using the polar decomposition of $\mathrm{GL}(2,\RR)$ \cite{Ruelle}. In the context of twists maps, the notion of torsion of an orbit has been considered by J.\;Mather \cite{Mather1} \cite{Mather2} and S.\;Angenent \cite{Angenent} and S.\;Crovisier \cite{Crovisier1,Crovisier2}; see below for more details. J.Mather calls it ``the amount of rotation of an orbit''. The torsion of the area probability measure is one of the quasi-morphisms considered by J.M.\;Gambaudo and \'E.\;Ghys to study the algebraic structure of the groups of area preserving diffeomorphisms of a surface; they call it ``the Ruelle number of a diffeomorphism'' \cite{GG:95}. A few years ago, T.\;Inaba and H.\;Nakayama have interpreted the torsion of an invariant probability measure as the difference of the volume of two subsets of a line bundle over the surface \cite{InabaNakayama}.

Existence results for orbits and/or invariant probability measures with zero torsion have been proved in several contextes. For area preserving twists maps of the compact annulus, well-ordered periodic orbits with a given rotation number can be found as critical points of a certain energy functionnal \cite{Mather1,AubryLeDaeron}. J.Mather and S.B.Angenent have related the Morse index of these critical points with the torsion of the corresponding orbit  \cite{Mather2}, \cite{Angenent}.The existence of periodic orbits with zero torsion for area preserving twists maps of the annulus follows (see e.g. \cite{Angenent} thm1). Actually, S.Crovisier has constructed orbits with arbitrary rotation number and zero torsion for any twist map of the annulus \cite{Crovisier2}. He subsequently used these orbits to exhibit generalized  Arnold's tongues for twist maps \cite{Crovisier1}. In other direction, S.Matsumoto and H.Nakayama have proved that every isotopic to the identity diffeomorphism of the torus $\TT^2$ has an invariant probability measure with zero torsion \cite{MatsumotoNakayama}.

The purpose of the present paper is to prove results which go in the opposite direction: we want to identify general situations where there exist orbits with non-zero torsion (note that the existence of an orbit with non-zero torsion is equivalent to the existence of an (ergodic) invariant probability measure with non-zero torsion, see section~\ref{s.definition-torsion}). In other words, we want to find general situations where the non-triviality of the dynamics of a surface diffeomorphism $f$, can be read on the action of the derivative of $f$ along a single orbit. 

\begin{rem}
Such results are typically useful to prove rigidity theorems for some types of group actions.  For example, J.\;Franks and M.\;Handel  have proved that every action of $\mathrm{SL}(3,\ZZ)$ on a surface by area preserving $C^1$-diffeomorphisms factors through a finite group (\cite{FranksHandel}). Their proof uses the fact that, for every area preserving diffeomorphism $f$ on a closed surface of negative Euler characteristic, either $f$ has trivial dynamics (there exits $n$ such that $f^n=\mathrm{Id}$), or there is a simple dynamical invariant which detects the non-triviality of the dynamics of $f$: there is an orbit with non-zero rotation vector, or there is a closed curve $\alpha$ such that the length of $f^n(\alpha)$ grows exponentially, or... Similarly, in order to prove that every action of $\ZZ^n$ on $\SS^2$ by area preserving $C^1$-diffeomorphisms has two global fixed points (provided that it is generated by diffeomorphisms that are $C^0$-close to the identity), the authors of \cite{BFLM} have used the fact that every area preserving diffeomorphism of $\SS^2$ has two fixed points $a, b$ and a recurrent point $c$ such that $c$ has a non-zero rotation number in the annulus $\SS^2 \setminus \{a,b\}$.
\end{rem}

Our first result concerns conservative diffeomorphisms of the disc: 

\begin{thm}\label{thm1}
Every area-preserving diffeomorphism of the disk $\D^{2}$ with compact support, which is not the identity, has an orbit with non-zero torsion.
\end{thm}

In the above statement, the disk $\DD^2$ is tacitly equipped with a trivialization of its tangent bundle. Observe that, since $\DD^2$ is simply connected and since we consider diffeomorphisms with compact support, the result does not depend on the choice of this trivialization. Also recall that, for diffeomorphisms of the disc with compact support, the torsion of an orbit does not depend on the choice of an isotopy.

\begin{rem}
Let $S$ be an orientable surface with zero genus, and $f$ be an area-preserving diffeomorphism of the surface $S$ with compact support. Then we can embed $S$ in $\DD^2$, extend $f$ to a diffeomorphism $\bar f$ of $\DD^2$ which is the identity on $\DD^2 \setminus  S$. Then, applying theorem \ref{thm1} to $\bar f$, we get an orbit with non-zero torsion. This orbit is automatically in $S$, since $\bar f=Id$ on $\DD^2 \setminus  S$. Observe nevertheless that, if $S$ is not simply connected, the torsion of an orbit of $f$ does depend on the choice of a trivialization of the tangent bundle of $S$. Moreover, there exist trivializations of the tangent bundle of $S$ which do not extend to trivializations of the tangent bundle of $\DD^2$ for any embedding of $S$ in $\DD^2$. Note, for example, that all the orbits of a rigid rotation of the annulus $\AA= \RR/\ZZ \times  \RR$ have zero torsion with respect to the natural trivialization of the tangent bundle of $\AA$ (\emph{i.e.} the trivialization induce by the canonical trivialization of the tangent bundle of $\RR \times \RR$).
\end{rem}

The existence of orbits with non-zero torsion will be obtained as a consequence of the existence of a recurrent orbit rotating around a fixed point (at a non-zero average speed). We will obtain the existence of such a recurrent orbit and such a fixed point as a consequence of a symplectic geometry result of C. Viterbo (\cite{Viterbo}). It could also be obtained as a consequence of P. Le Calvez's foliated equivariant version of Brouwer's plane translation theorem (\cite{LeCalvez}), together with a recent result of O. Jaulent which allows to apply Le Calvez's result to a diffeomorphism which has infinitely many fixed points (\cite{Jaulent}). 

Theorem~\ref{thm1} fails to be true if one replaces the disc $\DD^2$ by another surface. Indeed, for a surface which is not the disk, the recurrent orbits can rotate around the holes or the handles of the surface instead of rotating around the fixed points. A rigid rotation of $\T^{2}=\RR^2 / \ZZ^2$ is an example of an area-preserving diffeomorphism having only orbits with zero torsion (for the trivialization of the tangent bundle of $\T^2$ induced by the canonical trivialization of the tangent bundle of $\RR^2$). We will nevertheless prove a result of existence of orbits with non-zero torsion for some diffeomorphisms of $\T^{2}$; unlike the rigid rotations, these diffeomorphims will have orbits ''rotating around $\T^{2}$'' in different directions.

Let us  recall the definition of the rotation set of a torus homeomorphism. Let $f$ be a homeomorphism of the torus $\TT^2$ isotopic to the identity, and $\widetilde{f}:\RR^2\to\RR^2$ be a lift of $f$. Given a point $z \in\TT^2$ and a lift $\widetilde z\in\RR^2$ of $z$, let $\rho_n(\widetilde{f},z):=\frac{1}{n}(\widetilde{f}^n(\widetilde z)-\widetilde z)$ (this quantity does not depend on the choice of $\widetilde z$). If $\rho _n(\widetilde{f},z)$ converges towards $\rho(\widetilde{f},z)$ as $n$ goes to $+\infty$, we say that $\rho(\widetilde{f},z)$ is the \emph{rotation vector} of $z$ for $\widetilde{f}$. Note that, if $z$ is a periodic point, then the rotation vector $\rho(\widetilde{f},z)$ is well-defined, and has rational coordinates. Now, the \emph{rotation set} of $\widetilde{f}$ is 
$$\rho(f):=\bigcap _{n\geq 0} \overline{\{\rho _n(\widetilde{f},z),\;\; z \in \RR^2\}}.$$
 This is a convex compact of $\RR^2$ (\cite{Misu-Zie}). If $\widetilde{f}'$ is another lift of $f$, then there exists $(p,p') \in \ZZ^2$ such that $\widetilde{f}'=\widetilde{f}+(p,p')$, and thus, $\rho(\widetilde{f}')=\rho(\widetilde{f})+(p,p')$. We will speak of the rotation set of $f$: this is a subset of $\RR^2$, defined up to a translation (in particular, the fact that the rotation has non-empty interior does not depend on the choice of a lift of $f$).

We will prove the following result:

\begin{thm}\label{thm2}
Every diffeomorphism of $\T^{2}$, isotopic to the identity, and whose rotation set has non-empty interior, has
 an orbit with non zero torsion.
\end{thm}

In this statement, the torsion is measured with respect to the ``canonical'' trivialization of the tangent bundle of $\TT^2=\RR^2 / \ZZ^2$, \emph{i.e.} the trivialization induced by the affine structure of $\RR^2$. We have already pointed out that, for this trivialization, the torsion of an orbit does not depend on the choice of an isotopy joining the identity to $f$. Using a result of Franks (\cite{Franks}), the hypothesis of theorem \ref{thm2} may be replaced by the following one: $f$ has three periodic orbits whose rotation vectors are affinely independent. 

\subsection*{Comments and questions.} 

\begin{enumerate}
\item Matsumoto and Nakayama have proved that every diffeomorphism of the torus $\TT^2$ admits an invariant probability measure with zero torsion (\cite{MatsumotoNakayama}). It  seems to be unknown if one can find such a probability measure which is moreover ergodic. And thus, it seems to be unknown if one can find an orbit with zero torsion. On the contrary, observe that the existence of a measure with non-zero torsion automatically implies  (using the ergodic decomposition theorem) the existence of an ergodic measure with non-zero torsion, and thus (using Birkhoff's ergodic theorem) the existence of an orbit with non-zero torsion.

\item We do not know if, under the hypotheses of theorems~\ref{thm1} and~\ref{thm2}, it is possible to find \emph{periodic} orbits with non-zero torsion.

\item As we already mentionned, our proof of theorem \ref{thm1} relies on a symplectic geometry result due to Viterbo.
Using a recent result by Jaulent and the equivariant foliated Brouwer's translation theorem of Le Calvez instead of Viterbo's result, it is possible to prove that theorem \ref{thm1} is still valid for a diffeomorphism $f$ of $\DD^2$ which preserves a probability measure whose support is not contained in the fixed points set of $f$. 

\item Theorem~\ref{thm1} is of course false, if one simply drops the area-preservation hypothesis. A counter-example is obtained by considering a diffeomorphism which preserves each horizontal line. Nevertheless, we do not know if theorem~\ref{thm1} is true or false, if one assumes for example the existence of a non-wandering orbit (which is not a fixed point) instead of the preservation of some probability measure.

\item It seems reasonable to think that under the hypothesis of theorem \ref{thm1}, the set of points whose orbits have non-zero torsion should have positive area, but we were unfortunately not able to prove such a result.

\item We do not know if it is possible to prove a quantitative version of theorem \ref{thm2}. \emph{Is there a constant $C$ such that, every diffeomorphism $f$ of $\TT^2$ isotopic to the identity has an orbit of torsion at least $C.r$, where $r$ is the maximal radius of an euclidean ball contained in the rotation set of $f$?}

\item The arguments of the proof of theorem \ref{thm2} do not seem to allow to control the rotation vectors of the orbits with non-zero torsion that we construct. We do not know if, under the hypothesis of theorem \ref{thm2}, it is possible to construct an orbit with rotation vector $v$ and non-zero torsion, for every vector $v$ in the rotation set of a lift $\widetilde f$ of $f$.

\item  It has been proved by J. Llibre and R. MacKay that every $f$ of $\TT^2$, isotopic to the identity, whose rotation set has non-empty interior, has positive topological entropy. H. Einrich and N. Guelman and A. Larcanch\'e and I. Liousse have obtained a partial converse to this result: for every $C^{1+\alpha}$-diffeomorphism $f$ of $\TT^2$, isotopic to the identity, satisfying a transitivity assumption, if $f$ has positive topological entropy, then the rotation set of $f$ has non-empty interior. Combining Einrich-Guelman-Larcanch\'e-Liousse's result with our theorem~\ref{thm2}, we obtain that a $C^{1+\alpha}$-diffeomorphism of $\TT^2$, isotopic to the identity, satisfying a transitivity assumption, if $f$ has positive topological entropy, then $f$ has an orbit with non-zero torsion. We do not know if this result admits a converse:  \emph{does a $C^{1+\alpha}$ diffeomorphism of $\TT^2$, satisfying an appropriate transitivity assumption, and having an orbit with non-zero torsion, necessarly have positive topological entropy?}
\end{enumerate}   
              
\subsection*{Organization of the paper.} In section~\ref{s.definition-torsion}, we define precisely what we mean by the \emph{torsion} of an orbit or an invariant probability measure, and we give some basic properties of these notions. In section~\ref{s.linking-and-torsion}, we define the linking number of two orbits for a diffeomorphism of $\RR^2$ isotopic to the identity, and we prove that the existence of two orbits with non-zero linking number implies the existence of an orbit with non-zero torsion. Section~\ref{s.thm1} is devoted to the proof of theorem~\ref{thm1} and section~\ref{s.thm2} is devoted to the proof of theorem~\ref{thm2}.

\subsection*{Conventions and notations.} 
All along the paper, the vector space $\RR^2$ is endowed with its canonical euclidean norm, which we denote by $\|\cdot\|$. We denote by $\mathbb{S}^1$ the unit circle in $\RR^2$. We identify the universal cover of $\SS^1$ with the real line $\RR$ using the map $\pi:\RR\to\SS^1$  given by $\pi(\theta)=e^{i2\pi\theta}$. As a unit for angles, we use the full turn.

\subsection*{Acknowledgments}
The authors would like to thank S. Crovisier, P. Le Calvez, F. Le Roux and P. Py for helpful conversations. The authors are grateful to the \emph{Unit\'e de Recherche Math\'ematiques et Applications} of the \emph{Facult\'e des Sciences de Bizerte} and the \emph{Institut Pr\'eparatoire aux \'Etudes d'Ing\'enieur de Bizerte} for their support of the visits of the second author to Orsay, where this work was carried out.


  \section{ Torsion of an orbit or an invariant measure }
  \label{s.definition-torsion}

Let $S$ be a (not necessary compact) surface with trivializable tangent bundle; in other words, $S$ is an orientable surface with genus $0$ or $1$. We choose a trivialization of the tangent bundle of $S$, which identifies the tangent bundle $TS$ with $S\times\RR^2$. Thanks to this identification, the canonical euclidean norm of $\RR^2$ defines a riemannian metric on $S$. This allows us to speak of the unit tangent bundle of $S$, which we denote by $T^1S$.
By construction,  $T^1S$ is identified to~$S\times\SS^1$. Now, let $f$ be a ${\mathcal C}^{1}$-diffeomorphism of $S$ isotopic to the identity, and $I=(f_t)_{t\in [0,1]}$ be an isotopy joining the identity to $f$ in $\mathrm{Diff}^1(S)$. As usual, we define $f_t$ for every $t\in\RR$ by setting $f_t=f_{t-\lfloor t\rfloor} \circ f^{\lfloor t\rfloor}$. We denote by~$f_*$ the action of $f$ on the unit tangent bundle $T^1S\simeq S\times\SS^1$:
$$f_*\big(x,\xi\big) =  \left(f(x),\frac{df_t(x).\xi}{\|df_t (x).\xi\|}\right).$$

Now we define the torsion of an orbit of $f$. Let $x$ be a point in $S$, and $\xi$ be a unit tangent vector at $x$. For every $t$, we see $\frac{df_t(x).\xi}{\|df_t (x).\xi\|}$ as an element of $\SS^1$. We consider the map
$$\begin{array}{rrcl}
 v(I,x,\xi) : & \R & \longrightarrow &\SS^1\\
& t & \longmapsto & \displaystyle\frac{df_t(x).\xi}{\|df_t (x).\xi\|}
\end{array}.$$
We choose a continuous lift $\widetilde v(I,x,\xi):\RR\longrightarrow \RR $ of this map. For every $t \in \RR,$ the quantity $\widetilde{v}(I,x,\xi)(t)-\widetilde{v}(I,x,\xi)(0)$ does not depend on the choice of the lift  $ \widetilde{v}(I,x,\xi)$. We set
$$\mathrm{Torsion}_1(I,x,\xi)= \widetilde{v}(I,x,\xi)(1)-\widetilde{v}(I,x,\xi)(0).$$ 
More generally, for every $n\in \N\setminus\{0\}$, we set
$$ \mathrm{Torsion}_n(I,x,\xi) =
\frac{1}{n} \Big(\widetilde{v}(I,x,\xi)(n)-\widetilde{v}(I,x,\xi)(0) \Big)= \displaystyle\frac{1}{n}\sum^{n-1}_{k=0} \mathrm{Torsion}_1\left(I,f^k_*(x,\xi)\right).$$ 
Let $\xi'$ be another unit vector in $T^1_x S$;  since $df_n(x) :\xi \longmapsto v(I,x,\xi)(n)$ preserves the cyclic order on $\SS^1$, we have
$$\Big|\mathrm{Torsion}_n(I,x,\xi)- \mathrm{Torsion}_n(I,x,\xi')\Big|\leq  \frac{2}{n}.$$
This shows that  the quantity $\mathrm{Torsion}_n(I,x,\xi)$ has a limit for some $\xi\in T^1_x S$ when $n$ goes to $\infty$, if and only if it has a limit for every $\xi\in T^1_x S$, and that the numerical value of the limit does not depend on $\xi$.

\begin{defi}
Consider a point $x\in S$, and assume that the quantity $\mathrm{Torsion}_n (I,x,\xi) $ converges for some (or equivalently, for every) $\xi\in T^1 S$ when $n$ goes to $\infty$. Then we call \emph{the torsion of the orbit of $x$} (for the isotopy $I$) the quantity
$$\mathrm{Torsion}(I,x)=\lim_{n\to\infty} \mathrm{Torsion}_n(I,x,\xi).$$
\end{defi}

Now, let $\mu$ be  an $f$-invariant probability measure on $S$. We can lift $\mu$ to $\widetilde{\mu}$ an $f_*$-invariant probability measure on $T^1S \cong S \times \SS^1$. Assume that either $\mu$ or $f$ has compact support. Then, the function $(x,\xi) \mapsto  \mathrm{Torsion}_1 (I,x,\xi)$ is in $L^{\infty}(T^1S, \widetilde{\mu})$, since it is bounded on every compact subset of $T^1S$ and vanishes on the complement of the support of $f$. So, Birkhoff's ergodic theorem implies that the limit
 $$\mathrm{Torsion}(I,x)=\displaystyle\lim_{n\rightarrow +\infty}  \displaystyle\frac{1}{n} \sum^{n-1}_{k=0} \mathrm{Torsion}_1(I,f_*^k(x,\xi))$$ 
 exists for $\widetilde{\mu}$ almost every $(x,\xi)$, and hence, for $\mu$ almost every $x$ (since the convergence does not depend on $\xi$). Furthermore, and always by the Birkhoff ergodic theorem, the function $x \mapsto \mathrm{Torsion}(I,x)$ is $\mu-$integrable and 
$$\int_{S} \mathrm{Torsion}(I,x) d\mu (x)= \int_{T^1S} \mathrm{Torsion}_1(I,x,\xi) d\widetilde{\mu} (x,\xi).$$  
    
 \begin{defi}
Let $\mu$ be a $f$-invariant probability measure on $S$.  Assume that either $\mu$ or $f$ has compact support.  We call the \emph{torsion of the measure $\mu$} the quantity
$$\mathrm{Torsion}(I,\mu)=\int_{S} \mathrm{Torsion}(I,x) d\mu (x).$$
\end{defi}

Now, we give some relations between these different notions of torsion. 

\begin{lem}\label{lem1}
Let $(x_n,\xi_n)_{n\geq 0}$ be a sequence in some compact subset of~$T^1 S$. For every $n$,  let  $\alpha_n:=\mathrm{Torsion}_n(I,x_n,\xi_n)$. Then, every limit point of the sequence $(\alpha_n)_{n\geq0}$ is the torsion of an $f$-invariant probability measure whose support is contained in $\overline{\{x_n \mid  n \in \NN\}}.$
\end{lem}

\begin{proof}
 Let $\alpha = \lim_{i\rightarrow + \infty} \alpha_{n_i}$ be a limit point of the sequence $(\alpha_n)_{n\geq0}$. 
For every $i$, consider the probability measure $\widetilde{\mu}_i$ on $T^1 S$ defined by 
$$\widetilde{\mu}_i=\displaystyle\frac{1}{n_i}\displaystyle\sum^{n_i-1}_{k=0}\delta_{f_*^{k}(x_{n_i},\xi_{n_i})}.$$ 
Let $\widetilde{\mu}$ be a limit point of the sequence $(\widetilde{\mu}_i)_{i \in \N}$. This is a $f_*$-invariant probability  measure on~$T^1 S$. So the projection $\mu$ of $\widetilde\mu$  is a $f$-invariant probability measure on $S$, and one has  
$$\alpha = \mathop{\lim}_{i\to\infty}\alpha_{n_i}=\displaystyle\lim_{i\rightarrow\infty} \int_{T^1S} \mathrm{Torsion}_1(I,x,\xi) d\widetilde{\mu}_i (x,\xi)$$
$$ =\int_{T^1S} \mathrm{Torsion}_1(I,x,\xi) d\widetilde{\mu} (x,\xi)=\int_{S} \mathrm{Torsion}_1(I,x) d\mu (x)= \mathrm{Torsion}(I,\mu).$$
Moreover, the support of $\mu$ is obviously contained in $\overline{\{x_n \mid  n \in \NN\}}$. 
\end{proof}
 
 \begin{lem}
 \label{lem2}
For every $f$-invariant probability measure $\mu$, the torsion of the measure $\mu$ is a convex combination of torsions of orbits of points of $S$. 
\end{lem}

\begin{proof}
On the one hand, the map $\mu \mapsto \mathrm{Torsion}(I,\mu)$ is affine, so the ergodic decomposition theorem shows that the torsion of $\mu$  is a  convex combination of torsions of ergodic probability measure. On the other hand, it follows from Birkhoff's ergodic theorem that the torsion of an ergodic measure is the torsion of an orbit.  The lemma follows.
 \end{proof}

 Combining these two lemma, we obtain the following corollary:

 \begin{corol}
 \label{corol1}
 If there exists a sequence of integers $n_i \rightarrow + \infty$ and a sequence of points $(x_{n_i})$ in a compact of $S$ such that $\mathrm{Torsion}_{n_i}(I,x_{n_i}) \geq \epsilon$ for every $i$, then there exists a point $x \in S$ such that $\mathrm{Torsion}(I,x) \geq \epsilon.$
 \end{corol}

Now, we examine the dependance of the torsion of an orbit or an invariant probability measure with respect to the isotopy $I$. 

Let $I'$ be another isotopy joining the identity to $f$. Then, for every $(x,\xi)\in T^1S$, the quantities $\mathrm{Torsion}_1(I,x,\xi)$ and $\mathrm{Torsion}_1(I',x,\xi)$ differ by an integer. For continuity reasons, this integer does not depend on $x$ and $\xi$ (provided that $S$ is connnected).  It follows that there is an integer $k\in\ZZ$, such that $\mathrm{Torsion}(I',x)=\mathrm{Torsion}(I',x)+k$ for every point $x\in S$, and $\mathrm{Torsion}(I',\mu)=\mathrm{Torsion}(I',\mu)+k$ for every $f$-invariant probability measure $\mu$. Now, the integer $k$ obviously depends continuously on the isotopy $I'$. It follows that, for every point $x\in S$ and every $f$-invariant probability measure $\mu$, the quantities $\mathrm{Torsion}(I,x)$ and $\mathrm{Torsion}(I,\mu)$ depend only on the homotopy class of the isotopy $I$.

There exist several interesting situations where these quantities $\mathrm{Torsion}(I,x)$ and $\mathrm{Torsion}(I,\mu)$ depend only on $f$, and not on the choice of the isotopy $I$ joining the identity to $f$.

Let us first consider the case where $S$ is the torus $\TT^2= \RR^2 / \ZZ^2$ endowed with the canonical trivialization 
of its tangent bundle (\emph{i.e.} the trivialization induced by the affine structure of $\RR^2$). We know that $\mathrm{Diff}^1_0(\TT^2)$ retracts on the subgroup (isomorphic to $\TT^2$) made of the rigid rotations. Thus, if $I$ and $I'$ are two isotopies joining the identity to a diffeomorphism $f$ of $\TT^2$, then $I'^{-1}I$ is homotopic to a loop in the rigid rotations group. But, it is clear, that if $I''$ is a loop made of rigid rotations then, for every $(x, \xi) \in T^1 \TT^2$, we have $\mathrm{Torsion}_1(I'', x, \xi)=0$ (the rigid rotations are parallel with respect to the canonical trivialization of the tangent bundle). It  follows that the quantities $;\mathrm{Torsion}(I,x)$ and $\mathrm{Torsion}(I,\mu)$  do not depend on the choice of the isotopy $I$.

Now assume that $S$ is not a closed surface. Let $\mathrm{Diff}^1_c(S)$ be the set of the $C^1$-diffeomorphims of $S$ with compact supports, and $\mathrm{Diff}^1_{c,0}(S)$ the subset of $\mathrm{Diff}^1_c(S)$ made of the diffeomorphisms that are isotopic to the identity (\emph{via} an isotopy in $\mathrm{Diff}^1_c(S)$). We know that $\mathrm{Diff}^1_{c,0}(S)$ is contractible ( as a corollary of Kneser's theorem, see \cite{Kneser} or e.g.~\cite[th\'eor\`eme 2.9]{LeRoux}). Hence, if $f \in \mathrm{Diff}^1_{c,0}(S)$ and if we consider only isotopies $I$ joining the identity to $f$ in $\mathrm{Diff}^1_c(S)$, then the quantities $\mathrm{Torsion}(I,x)$ and $\mathrm{Torsion}(I,\mu)$ depend only on $f$ and not on the choice of $I$.

When the quantities  $\mathrm{Torsion}(I,x)$ and $\mathrm{Torsion}(I,\mu)$ do not depend on the isotopy $I$, we will denote them by $\mathrm{Torsion}(f,x)$ and $\mathrm{Torsion}(f,\mu).$

\medskip

In general, the torsion of an orbit of $f$ or a $f$-invariant probability measure also depends on the trivialization of the tangent bundle of $S$ which is used to identify $T^1 S$ with $S\times \SS^1$. Nevertheless, it is quite easy to check that, if $f$ has compact support in $S$, then two trivializations that are homotopic yield the same value for  the torsion of an orbit of $f$ or a $f$-invariant probability measure. In particular, if $S$ is the disc $\DD^2$ and $f$ has compact support, then the torsion of an orbit of $f$ and the torsion of a $f$-invariant probability measure do not depend on the trivialization of the tangent bundle of $S$. Although this fact is worth to be pointed out, we shall not use it strictly speaking.

\medskip

We conclude this section by a remark, which will be useful for the proof of our theorem~\ref{thm2}:

\begin{rem}
\label{rem1}
Assume $\RR^2$ and $\TT^2=\RR^2/\ZZ^2$ to be equipped with the canonical trivializations of their tangent bundles.
Let $I$ be an isotopy joining the identity to a diffeomorphism $f$ in $\mathrm{Diff}^1(\TT^2)$, and $\widetilde I$ be a lift of the isotopy $I$ in $\mathrm{Diff}^1(\RR^2)$. Let $x$ be a point in $\TT^2$. If there is a lift $\widetilde x\in\RR^2$ of $x$ such that the quantity $\mathrm{Torsion}(\widetilde I,\widetilde x)$ is well-defined, then the quantity $\mathrm{Torsion}(I,x)$ is also well-defined, and $\mathrm{Torsion}(\widetilde I,\widetilde x)=\mathrm{Torsion}(I,x)$. This follows immediatly from the definitions.
\end{rem}


  \section{Linking and torsion}
  \label{s.linking-and-torsion}

In this section, we consider a diffeomorphism $f$ of $\RR^2$ which we assume to be isotopic to the identity, and an isotopy $I=(f_t)_{t\in [0,1]}$ joining the identity to $f$. The plane $\RR^2$ on which the diffeomorphism $f$ acts should be considered as an affine plane rather than a vector space. We will denote by $d(\cdot,\cdot)$ the euclidean distance on this affine plane, which is induced by the canonical euclidean norm $\|\cdot\|$ on the underlying vector space. \emph{All along this section, $\RR^2$ is endowed with the canonical trivialization of its tangent bundle.}

\bigskip
 
Let us first define the \emph{linking number} of a pair of orbits of $f$. Let $x,y$ be two distinct points in $\RR^2$. We consider the map
$$\begin{array}{rrcl}
 v(I,x,y) : & \R & \longrightarrow &\SS^1\\
& t & \longmapsto & \displaystyle\frac{df_t(x)-df_t(y)}{\|df_t (x)-df_t(y)\|}
\end{array}.$$
We choose a continuous lift $\widetilde v(I,x,y):\RR\longrightarrow \RR $ of $v(I,x,y)$ (we recall that $\RR$ is seen as the universal cover of $\SS^1$, thanks to the map $\pi:\RR\to\SS^1$ given by $\pi(\theta)=e^{i2\pi\theta}$). For every $t \in \RR,$ the quantity $\widetilde{v}(I,x,y)(t)-\widetilde{v}(I,x,y)(0)$ does not depend on the choice of this lift. For $n\in \N\setminus\{0\}$, we set
$$ \mathrm{Linking}_n(I,x,y) :=
\frac{1}{n} \Big(\widetilde{v}(I,x,y)(n)-\widetilde{v}(I,x,y)(0) \Big)= \displaystyle\frac{1}{n}\sum^{n-1}_{k=0} \mathrm{Linking}_1\left(I,f^k(x),f^k(y)\right).$$ 

\begin{defi}
If $\mathrm{Linking}_n (I,x,y) $ converges when $n$ goes to $\infty$, then we call \emph{linking number of the orbits of $x$ and $y$} (for the isotopy $I$) the quantity
$$\mathrm{Linking}(I,x,y)=\lim_{n\to\infty} \mathrm{Linking}_n(I,x,y).$$
\end{defi}

The purpose of the section is to prove the following proposition:

\begin{pro}
\label{prop.linking-implies-torsion}
 If $f$ has two orbits with non-zero linking number, then $f$ has an orbit with non-zero torsion.
\end{pro}

\begin{rem}
Since we are dealing with diffeomorphisms with non-compact support, the linking number of a pair of orbits, as well as the torsion of an orbit, may depend on the choice of a trivialization of the tangent bundle of $\RR^2$. We recall that, all along this section, the plane $\RR^2$ is endowed with the canonical trivialization of its tangent bundle.
\end{rem}

The proof of proposition~\ref{prop.linking-implies-torsion} was sketched during a discussion between by the first author, F.~Le Roux and P.~Py several years ago. Observe proposition is intuitively obvious. Indeed, saying that the orbits of $x$ and $y$ have a non-zero linking number means that the segment line $[f_t(x),f_t(y)]$ turns with a non-zero average speed, when $t$ increases. And this should clearly imply that, at least for \emph{some} point $z$ of the segment line $[x,y]$, the image under $df_t(z)$ of the vector which is tangent to the segment $[x,y]$ must turn at a non-zero average speed, when $t$ increases. Nevertheless, turning this intuition into a formal proof appears to be quite delicate. We start by giving a quantitative version of the proposition~\ref{prop.linking-implies-torsion}~: 

\begin{lem}
\label{lemma.linking-implies-torsion}
Suppose that there exist two points $x,y\in\RR^2$ and an integer $n>0$ such that $\mathrm{Linking}_n(I,x,y)\neq 0$. Then, there exist a point $z\in\RR^2$ et a vector $\xi\in T_z \RR^2$ such that  
$$\left |\mathrm{Torsion}_n(I,z,\xi)\right |\geq \frac{1}{3} \left |\mathrm{Linking}_n(I,x,y)\right |-\frac{1}{n}.$$
\end{lem}

Proposition~\ref{prop.linking-implies-torsion} follows from lemma~\ref{lemma.linking-implies-torsion} and corollary~\ref{corol1}; so we are left to prove lemma~\ref{lemma.linking-implies-torsion}.

\begin{proof}[Proof of lemma~\ref{lemma.linking-implies-torsion}]
Let $x,y$ be two distinct points in $\RR^2$, and $n$ be a positive integer such that $\mathrm{Linking}_n(I,x,y)\neq 0$.
Let $\epsilon:=\mathrm{Linking}_n(I,x,y)$ and assume that $\epsilon$ is positive (the proof is similar in the case where it is negative). For $s \in [0,1]$, let  $z(s)=(1-s)x+sy$. Let $\xi:=\frac{y-x}{\|y-x\|}$. Using affine structure of~$\RR^2$, we will see the vector $\xi$ as an element of $T_{z(s)}\RR^2$ for every $s\in [0,1]$.  In order to prove the lemma, we will find $s_0\in [0,1]$ such that  $\mathrm{Torsion}_n(f,z(s_0),\xi)\geq \frac{\epsilon}{3}-\frac{1}{n}$. 

If  $\mathrm{Torsion}_n(I,x,\xi)\geq \frac{\epsilon}{3}$, then we take $s_0=0$ and we are done. Hence, in the remainder of the proof, we will assume that 
\begin{equation}
\label{e.hypothese-torsion-x}
\mathrm{Torsion}_n(I,x,\xi) \leq  \frac{\epsilon}{3}.
\end{equation}

Using the affine structure of $\RR^2$, we can define a map 
\begin{eqnarray*}
u \; : \; [0,1]\times\RR & \longrightarrow &\SS^1  \\
 (s,t) & \longmapsto & \frac{f_t(z(s))-f_t(x)}{\|f_t(z(s))-f_t(x)\|} \quad \mbox{ if } s>0, \\
 (0,t) & \longmapsto & \frac{df_t(z(s)).\xi}{\|df_t(z(s))\xi\|}.
\end{eqnarray*}
This map is continuous since $f_t$ is assumed to be $C^1$ and to depend continuously on $t$. So we may consider a  continuous lift $\widetilde u : [0,1]\times\RR\rightarrow\RR$ of the map $u$. This lift $\widetilde u $ is well defined up to the addition of an integer. 

\begin{claim} 
\label{c.1}
One has $\widetilde u(1,n)-\widetilde u(0,n) \geq \frac{2n\epsilon}{3}$.
\end{claim}

\begin{proof}
Observe that $\widetilde u(1,n)-\widetilde u(1,0)= n.\mathrm{Linking}_n(I,x,y)$ and $\widetilde u(0,n)-\widetilde u(0,0)=n.\mathrm{Torsion}_n(I,x,\xi)$. Also observe that $\widetilde u(1,0)-\widetilde u(0,0)=0$ since  $u(s,0)=\xi$ for any $s\in [0,1]$. Now write
\begin{eqnarray*}
\widetilde u(1,n)-\widetilde u(0,n) & = & 
\Big(\widetilde u(1,n)-\widetilde u(1,0)\Big) + \Big(\widetilde u(1,0)-\widetilde u(0,0)\Big) + \Big(\widetilde u(0,0)-\widetilde u(0,n)\Big).\\
& = & n.\mathrm{Linking}_n(I,x,y) - n.\mathrm{Torsion}_n(I,x,\xi)
\end{eqnarray*}
Recall that $\epsilon$ is by definition equal to $\mathrm{Linking}_n(I,x,y)$. Also recall that we are assuming that $\mathrm{Torsion}_n(I,x,\xi)$ is smaller than $\epsilon/3$ (hypothesis~\eqref{e.hypothese-torsion-x}). The claim follows. 
\end{proof} 

Now, we consider 
$$s_0:=\inf\left\{s\in [0,1] \mbox{ such that } \widetilde u(s,n)-\widetilde u(0,n) \geq \frac{2n\epsilon}{3}\right\}.$$
Claim~1 shows that $s_0$ is well-defined, and hypothesis~\eqref{e.hypothese-torsion-x} shows that $s_0$ is positive.\
Now, we introduce the continuous map
$$\begin{array}{ccrcll}
v & : & [0,1]\times\RR & \rightarrow & \SS^1 & \\
&& (s,t) & \mapsto & \frac{df_t(z(s)).\xi}{\|df_t(z(s).\xi\|}.
\end{array}$$
It is lift to a continuous map $\widetilde v : [0,1]\times\RR\rightarrow\RR$  satisfying $\pi\circ\widetilde v= v$, well defined by addition of an integer.

\begin{claim}
\label{c.2}
One has $\widetilde v(s_0,n)-\widetilde v(0,n)\geq \frac{2n\epsilon}{3} - \frac{3}{4}$.
\end{claim}

\begin{proof}
We consider the curve $\alpha:[0,s_0]\to \RR^2$ defined by $\alpha(s):=f^n(z(s))-f^n(x)$. We observe that, for every $s\in [0,s_0]$, 
$$v(s,n)=\frac{\frac{d\alpha}{ds}(s)}{\left\|\frac{d\alpha}{ds}(s)\right\|}\quad\mbox{and}\quad u(s,n)=\frac{\alpha(s)}{\|\alpha(s)\|}.$$
So, the quantity $\widetilde v(s_0,n)-\widetilde v(0,n)$ is the number of turns made by the tangent vector of the curve $\alpha$, and the quantity $\widetilde u(s_0,n)-\widetilde u(0,n)$ is the number of turns made by the curve $\alpha$ around the origin of $\RR^2$. To compare these two quantities, we will construct a simple closed curve (by concatenating a lift of $\alpha$ with a curve made of a controlled number of segments and circle arcs), and use the fact that the tangent vector of a simple closed curve makes $+1$ turn or $-1$ turn, when one goes around this curve once.

We introduce the annulus $\AA:=[0,+\infty)\times \SS^1$ obtained by blowing up the origin of  $\RR^2$. We will denote by $p:\AA\to\RR^2$ the natural projection given by  $p(r,e^{2i\pi\theta})=re^{2i\pi\theta}$. The strip $\widetilde\AA:=[0,\infty)\times\RR$ will be seen as the universal covering of $\AA$, the covering map  $\Pi:\widetilde\AA\to\AA$ being given by $\Pi(r,\theta)=(r,\pi(\theta))=(r,e^{2i\pi\theta})$. We will denote by $p_\theta:\widetilde\AA\to\RR$ the projection on the second coordinate. The inclusion of $\widetilde\AA=[0,+\infty)\times\RR$ in the affine space $\RR\times\RR$ provides a natural trivialization of the tangent bundle of $\widetilde\AA$.

\begin{rem}
\label{r.difference-trivializations}
Caution~! This trivialization of the tangent bundle of $\widetilde\AA$ is not the same that the one obtained, by pulling back by $p\circ \Pi:\widetilde\AA\to\RR^2$ the natural trivialization of the tangent bundle of the affine plane $\RR^2$, where the image of the curve $\alpha$ lies. More precisely, for $(r,\theta)\in\widetilde\AA$, using the trivializations of the tangent bundles of $\widetilde\AA$ and $\RR^2$, we can consider the differential map $d(p\circ \Pi)(r,\theta)$ as a linear map from $\RR^2$ to $\RR^2$; this linear map is \emph{not} the identity; it is a rotation of angle $\theta$ (recall that we use the full turn as a unit for angles).
\end{rem}

\begin{figure}[ht]
\centerline{\ifx\JPicScale\undefined\def\JPicScale{1}\fi
\psset{unit=\JPicScale mm}
\psset{linewidth=0.3,dotsep=1,hatchwidth=0.3,hatchsep=1.5,shadowsize=1,dimen=middle}
\psset{dotsize=0.7 2.5,dotscale=1 1,fillcolor=black}
\psset{arrowsize=1 2,arrowlength=1,arrowinset=0.25,tbarsize=0.7 5,bracketlength=0.15,rbracketlength=0.15}
\begin{pspicture}(0,0)(118.77,94.26)
\pspolygon[linewidth=0.1,linestyle=none,fillcolor=lightgray,fillstyle=solid](84,8)(84,78)(108,78)(108,8)
\pspolygon[linewidth=0.1,linecolor=yellow,linestyle=none,fillcolor=lightgray,fillstyle=solid](8,52)(52,52)(52,8)(8,8)
\rput{0}(30,30){\psellipse[linewidth=0.1,linestyle=dashed,dash=1 1](0,0)(10,10)}
\psline[linewidth=0.1,linecolor=blue,fillcolor=cyan,fillstyle=solid]{->}(40,30)(44,30)
\psline[linewidth=0.1,linecolor=blue,fillcolor=cyan,fillstyle=solid]{->>}(40,30)(40,34)
\psline[linewidth=0.1,linecolor=green,fillcolor=cyan,fillstyle=solid]{->}(30,40)(30,44)
\psline[linewidth=0.1,linecolor=green,fillcolor=cyan,fillstyle=solid]{->>}(30,40)(26,40)
\psline[linewidth=0.1,linecolor=red,fillcolor=cyan,fillstyle=solid]{->>}(20,30)(20,26)
\psline[linewidth=0.1,linecolor=red,fillcolor=cyan,fillstyle=solid]{->}(20,30)(16,30)
\psline[linewidth=0.1,linecolor=yellow,fillcolor=cyan,fillstyle=solid]{->}(30,20)(30,16)
\psline[linewidth=0.1,linecolor=yellow,fillcolor=cyan,fillstyle=solid]{->>}(30,20)(34,20)
\psline[linewidth=0.25,fillcolor=cyan,fillstyle=solid]{->}(78,30)(58,30)
\psline[linewidth=0.1,fillcolor=cyan,fillstyle=solid]{->}(84,8)(84,80)
\psline[linewidth=0.1,fillcolor=cyan,fillstyle=solid]{->}(83,24)(110,24)
\rput{0}(30,30){\psellipse[linewidth=0.1,fillstyle=solid](0,0)(0.44,0.44)}
\psline[linewidth=0.1,linestyle=dashed,dash=1 1,fillcolor=cyan,fillstyle=solid](94,8)(94,78)
\psline[linewidth=0.1,linecolor=blue,fillcolor=cyan,fillstyle=solid]{->}(94,72)(98,72)
\psline[linewidth=0.1,linecolor=blue,fillcolor=cyan,fillstyle=solid]{->>}(94,24)(94,28)
\psline[linewidth=0.1,linecolor=green,fillcolor=cyan,fillstyle=solid]{->>}(94,36)(94,40)
\psline[linewidth=0.1,linecolor=green,fillcolor=cyan,fillstyle=solid]{->}(94,36)(98,36)
\psline[linewidth=0.1,linecolor=red,fillcolor=cyan,fillstyle=solid]{->>}(94,48)(94,52)
\psline[linewidth=0.1,linecolor=red,fillcolor=cyan,fillstyle=solid]{->}(94,48)(98,48)
\psline[linewidth=0.1,linecolor=yellow,fillcolor=cyan,fillstyle=solid]{->>}(94,60)(94,64)
\psline[linewidth=0.1,linecolor=yellow,fillcolor=cyan,fillstyle=solid]{->}(94,60)(98,60)
\rput[bl](110,25){$r$}
\rput[bl](63,31){$p\circ\Pi$}
\rput[bl](80,23){$0$}
\rput[bl](77,10){$-\frac{1}{4}$}
\rput[bl](117.77,54.26){}
\psline[linewidth=0.1,fillcolor=cyan,fillstyle=solid](83,12)(85,12)
\psline[linewidth=0.1,fillcolor=cyan,fillstyle=solid](83,36)(85,36)
\rput[bl](79,34){$\frac{1}{4}$}
\rput[bl](118.77,94.26){}
\rput[bl](47,48){$\mathbb{R}^2$}
\rput[bl](50,49){}
\rput[bl](105,74){$\widetilde{\mathbb{A}}$}
\psline[linewidth=0.1,fillcolor=cyan,fillstyle=solid](83,48)(85,48)
\psline[linewidth=0.1,fillcolor=cyan,fillstyle=solid](83,72)(85,72)
\psline[linewidth=0.1,fillcolor=cyan,fillstyle=solid](83,60)(85,60)
\rput[bl](79,58){$\frac{3}{4}$}
\rput[bl](80,71){$1$}
\rput[bl](84,80){$\theta$}
\psline[linewidth=0.1,linecolor=yellow,fillcolor=cyan,fillstyle=solid]{->>}(94,12)(94,16)
\psline[linewidth=0.1,linecolor=yellow,fillcolor=cyan,fillstyle=solid]{->}(94,12)(98,12)
\psline[linewidth=0.1,linecolor=blue,fillcolor=cyan,fillstyle=solid]{->}(94,24)(98,24)
\psline[linewidth=0.1,linecolor=blue,fillcolor=cyan,fillstyle=solid]{->>}(94,72)(94,76)
\end{pspicture}}
\caption{\label{f.difference-trivializations}The trivializations of the tangent bundles of $\RR^2$ and $\widetilde{\mathbb{A}}$.}
\end{figure}
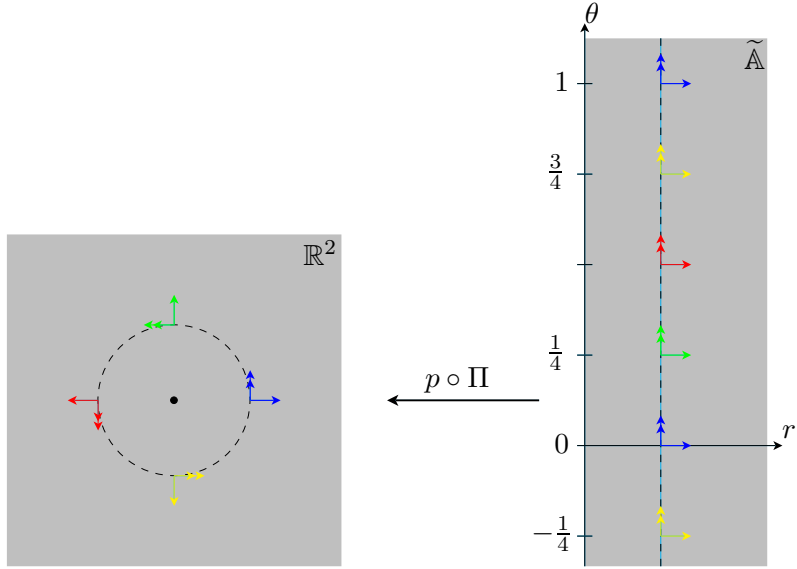

The curve $\alpha:[0,s_0]\to\RR^2$ can be lifted to a (uniquely defined) continuous curve  $\widehat\alpha:[0,s_0]\to\AA$, satisfying $p\circ \widehat\alpha=\alpha$. This curve $\widehat\alpha$ is given by the  formula $\widehat\alpha(s)=(\|f_n(z(s)-f_n(x)\|,u(s,n))$. Then the curve $\widehat\alpha:[0,s_0]\to\AA$ can be lifted to a continuous curve  $\widetilde\alpha:[0,s_0]\to \widetilde\AA$, satisfying $\Pi\circ \widetilde\alpha=\alpha$, given by the formula $\widetilde\alpha(s)=(\|f_n(z(s)-f_n(x)\|,\widetilde u(s,n))$. Hence, one has $p_\theta(\widetilde\alpha(s))=\widetilde u(s,n)$ for any $s\in [0,1]$. 

Remark that $\alpha$ is a simple $C^1$ curve; indeed, up to a translation, it is the image  of the segment line $s\mapsto z(s)$ under the diffeomorphism $f_n$. So, $\widetilde \alpha$ is a  simple $C^1$ curve as well. On the other hand, the definition of $s_0$ and the relation $p_\theta\circ \widetilde\alpha(s)=\widetilde u(s,n)$ imply that we have  
$$p_\theta(\widetilde\alpha(s_0))=p_\theta(\widetilde\alpha(0))+\frac{2n\epsilon}{3}\quad\quad\mbox{et}\quad\quad p_\theta(\widetilde\alpha(s))<p_\theta(\widetilde\alpha(0))+\frac{2n\epsilon}{3}\mbox{ for }s\in [0,s_0).$$
Hence, the curve $\widetilde \alpha$ is contained in the quarter of the plane 
$$Q=\{(r,\theta) \mid r \geq 0 \mbox{ et } \theta \leq p_\theta(\widetilde\alpha(0))+ \frac{2n\epsilon}{3}\},$$ 
and joins the point $\widetilde\alpha(0)$, which is situated on the  horizontal part of the boundary of $Q$, to the point $\widetilde\alpha(s_0)$ which is situated on the vertical part of the boundary of $Q$. It is so easy to construct a  simple arc  $\widetilde\beta$ in $\RR^2\setminus \mathrm{int}(Q)$ joining the extremities of $\widetilde\alpha$, such that $\widetilde\alpha\cup\widetilde\beta$ is a $C^1$ negatively orientated simple closed curve, and such that the tangent vector of $\widetilde\beta$ makes between $-\frac{5}{4}$ turns and $-\frac{1}{4}$ turns, when we cover $\widetilde\beta$. We can, for example, construct $\widetilde\beta$ as union of three circle arcs and two segment lines (see figure~\ref{f.getting-a-simple-closed-curve}). As $\widetilde\alpha\cup\widetilde\beta$ is a $C^1$ negatively orientated simple closed curve, the tangent vector of this curve makes  exactly $-1$ turn when one runs around the curve. Thus, the tangent vector of $\widetilde\alpha$ makes between $-\frac{3}{4}$ turn and $+\frac{1}{4}$ turn, when the parameter $s$ runs from $0$ to $s_0$. According to the remark~\ref{r.difference-trivializations}, and since  $p_\theta\circ\widetilde\alpha(s_0)-p_\theta\circ\widetilde\alpha(0)= \frac{2n\epsilon}{3}$, this means that, for the canonical trivialization  of the tangent bundle of $\RR^2$, the tangent vector of $\alpha=p\circ\Pi(\widetilde\alpha)$ makes between $\frac{2n\epsilon}{3}-\frac{3}{4}$ turns and $\frac{2n\epsilon}{3}+\frac{1}{4}$ turns, when $s$ runs from $0$ to $s_0$. . Recall, now, that this number of turns is equal to $\widetilde v(s_0,n)-\widetilde v(0,n)$. This completes the proof of claim~2.
\end{proof}

\begin{figure}[ht]
\label{f.getting-a-simple-closed-curve}
\centerline{\ifx\JPicScale\undefined\def\JPicScale{1}\fi
\psset{unit=\JPicScale mm}
\psset{linewidth=0.3,dotsep=1,hatchwidth=0.3,hatchsep=1.5,shadowsize=1,dimen=middle}
\psset{dotsize=0.7 2.5,dotscale=1 1,fillcolor=black}
\psset{arrowsize=1 2,arrowlength=1,arrowinset=0.25,tbarsize=0.7 5,bracketlength=0.15,rbracketlength=0.15}
\begin{pspicture}(0,0)(86,146.56)
\pspolygon[linewidth=0,linestyle=none,fillcolor=lightgray,fillstyle=solid](70.47,128.07)(70.94,3.13)(22,3)(21.53,127.93)
\psline[linewidth=0.2]{<-}(21.82,146.56)(21.82,2.8)
\psline[linewidth=0.2]{->}(22,32)(85,32)
\pscustom[linewidth=0.2,linecolor=red]{\psbezier(43.96,127.76)(43.96,125.89)(43.43,123.9)(42.18,121.15)
\psbezier(40.93,118.4)(41.09,117.81)(42.71,119.19)
\psbezier(44.34,120.56)(45.89,119.57)(47.89,115.89)
\psbezier(49.89,112.2)(48.51,111.55)(43.29,113.72)
\psbezier(38.06,115.89)(35.97,116.22)(36.32,114.82)
\psbezier(36.67,113.42)(37.78,112.59)(40.02,112.05)
\psbezier(42.26,111.5)(43.73,110.15)(44.92,107.52)
\psbezier(46.1,104.9)(46.26,102.86)(45.45,100.74)
\psbezier(44.64,98.61)(43.76,97.03)(42.51,95.47)
\psbezier(41.26,93.91)(39.7,93.16)(37.33,92.97)
\psbezier(34.96,92.78)(33.7,93.77)(33.14,96.27)
\psbezier(32.57,98.77)(33.14,100.54)(35.01,102.17)
\psbezier(36.89,103.79)(38.07,104.04)(38.94,102.97)
\psbezier(39.82,101.91)(39.9,101.08)(39.21,100.2)
\psbezier(38.52,99.33)(37.79,99.11)(36.79,99.49)
\psbezier(35.79,99.86)(35.36,99.3)(35.36,97.61)
\psbezier(35.36,95.92)(36.22,95.58)(38.22,96.45)
\psbezier(40.22,97.33)(41.38,98.45)(42.06,100.2)
\psbezier(42.75,101.95)(41.89,103.69)(39.2,106.01)
\psbezier(36.51,108.32)(34.32,107.78)(31.88,104.22)
\psbezier(29.45,100.66)(29.07,97.36)(30.63,93.24)
\psbezier(32.19,89.11)(34.74,87.9)(39.11,89.22)
\psbezier(43.49,90.53)(45.76,90.1)(46.7,87.79)
\psbezier(47.64,85.48)(48.07,83.96)(48.13,82.71)
\psbezier(48.19,81.46)(47.31,78.35)(45.19,72.35)
\psbezier(43.06,66.34)(43.22,62.89)(45.72,60.83)
\psbezier(48.22,58.76)(50.34,58.36)(52.78,59.49)
\psbezier(55.22,60.61)(56.66,62.05)(57.6,64.28)
\psbezier(58.54,66.51)(58.84,68.11)(58.61,69.64)
\psbezier(58.38,71.16)(58.14,72.07)(57.81,72.66)
\psbezier(57.47,73.26)(56.98,73.63)(56.17,73.91)
\psbezier(55.36,74.18)(54.6,74.29)(53.63,74.26)
\psbezier(52.66,74.23)(51.97,74.05)(51.34,73.64)
\psbezier(50.71,73.23)(50.12,72.3)(49.38,70.56)
\psbezier(48.63,68.81)(49.01,67.51)(50.63,66.23)
\psbezier(52.26,64.95)(53.16,64.86)(53.66,65.93)
\psbezier(54.15,67)(54.08,67.66)(53.44,68.12)
\psbezier(52.79,68.58)(52.32,68.67)(51.89,68.42)
\psbezier(51.45,68.17)(51.09,68.24)(50.68,68.64)
\psbezier(50.28,69.05)(50.25,69.42)(50.6,69.9)
\psbezier(50.94,70.37)(51.58,70.57)(52.73,70.57)
\psbezier(53.87,70.57)(54.66,70.1)(55.34,69.02)
\psbezier(56.03,67.93)(56.12,66.99)(55.67,65.91)
\psbezier(55.21,64.82)(54.82,64.07)(54.36,63.42)
\psbezier(53.91,62.77)(53.12,62.4)(51.75,62.18)
\psbezier(50.37,61.96)(49.39,61.96)(48.48,62.18)
\psbezier(47.56,62.4)(47.07,62.96)(46.85,64.05)
\psbezier(46.62,65.13)(46.52,66.07)(46.52,67.15)
\psbezier(46.52,68.24)(46.72,69.45)(47.17,71.19)
\psbezier(47.63,72.93)(48.41,74.15)(49.79,75.23)
\psbezier(51.16,76.32)(53.02,76.6)(56,76.16)
\psbezier(58.97,75.73)(60.54,74.98)(61.23,73.68)
\psbezier(61.92,72.37)(62.41,70.69)(62.86,68.08)
\psbezier(63.32,65.47)(63.02,63.7)(61.88,62.18)
\psbezier(60.73,60.65)(59.26,59.26)(56.98,57.52)
\psbezier(54.69,55.78)(53.12,55.03)(51.75,55.03)
\psbezier(50.37,55.03)(49.1,54.94)(47.5,54.72)
\psbezier(45.9,54.5)(44.52,53.76)(42.92,52.24)
\psbezier(41.32,50.71)(40.14,48.48)(39,44.78)
\psbezier(37.85,41.08)(37.95,38.65)(39.32,36.7)
\psbezier(40.69,34.74)(42.16,33.72)(44.23,33.28)
\psbezier(46.29,32.85)(47.86,32.75)(49.46,32.97)
\psbezier(51.06,33.19)(52.24,33.47)(53.38,33.9)
\psbezier(54.52,34.34)(55.6,35.08)(56.98,36.39)
\psbezier(58.35,37.69)(59.14,39.37)(59.59,41.98)
\psbezier(60.05,44.59)(59.16,46.36)(56.65,47.89)
\psbezier(54.13,49.41)(52.27,49.69)(50.44,48.82)
\psbezier(48.61,47.94)(47.23,46.92)(45.86,45.4)
\psbezier(44.49,43.87)(44.1,42.48)(44.56,40.74)
\psbezier(45.01,39)(45.9,38.16)(47.5,37.94)
\psbezier(49.1,37.72)(50.27,37.91)(51.42,38.56)
\psbezier(52.56,39.21)(53.15,40.05)(53.38,41.36)
\psbezier(53.61,42.66)(53.42,43.41)(52.73,43.85)
\psbezier(52.05,44.28)(51.46,44.19)(50.77,43.54)
\psbezier(50.08,42.88)(49.98,42.42)(50.44,41.98)
\psbezier(50.89,41.55)(50.89,41.08)(50.44,40.43)
\psbezier(49.98,39.77)(49.49,39.68)(48.8,40.12)
\psbezier(48.12,40.55)(47.82,41.02)(47.82,41.67)
\psbezier(47.82,42.32)(47.92,42.88)(48.15,43.54)
\psbezier(48.38,44.19)(48.87,44.66)(49.79,45.09)
\psbezier(50.7,45.53)(51.68,45.62)(53.05,45.4)
\psbezier(54.42,45.18)(55.11,44.62)(55.34,43.54)
\psbezier(55.57,42.45)(55.57,41.51)(55.34,40.43)
\psbezier(55.11,39.34)(54.72,38.69)(54.03,38.25)
\psbezier(53.35,37.82)(52.27,37.35)(50.44,36.7)
\psbezier(48.61,36.04)(47.43,35.95)(46.52,36.39)
\psbezier(45.6,36.82)(44.92,37.2)(44.23,37.63)
\psbezier(43.55,38.07)(43.15,39.18)(42.92,41.36)
\psbezier(42.69,43.54)(42.89,45.03)(43.57,46.33)
\psbezier(44.26,47.63)(45.24,48.66)(46.84,49.75)
\psbezier(48.44,50.83)(49.82,51.49)(51.42,51.93)
\psbezier(53.02,52.36)(54.69,52.64)(56.98,52.86)
\psbezier(59.26,53.08)(60.93,52.52)(62.53,51)
\psbezier(64.13,49.47)(65.02,47.89)(65.48,45.71)
\psbezier(65.93,43.53)(66.13,41.76)(66.13,39.81)
\psbezier(66.13,37.85)(65.74,36.45)(64.82,35.15)
\psbezier(63.9,33.84)(62.53,32.81)(60.25,31.73)
\psbezier(57.96,30.64)(56,30.17)(53.71,30.17)
\psbezier(51.42,30.17)(49.95,30.17)(48.81,30.17)
\psbezier(47.66,30.17)(46.58,29.99)(45.21,29.55)
\psbezier(43.84,29.12)(43.35,28.28)(43.58,26.76)
\psbezier(43.8,25.23)(44.29,24.4)(45.21,23.96)
\psbezier(46.13,23.53)(46.82,23.62)(47.5,24.27)
\psbezier(48.19,24.92)(48.19,25.39)(47.5,25.82)
\psbezier(46.82,26.26)(46.82,26.72)(47.5,27.38)
\psbezier(48.19,28.03)(48.97,27.85)(50.1,26.77)
\psbezier(51.23,25.68)(51.71,24.17)(51.71,21.72)
\psbezier(51.71,19.27)(48.43,20.11)(40.77,24.52)
\psbezier(33.1,28.93)(30.55,29.38)(32.27,26.02)
\psbezier(33.98,22.66)(34.43,20.41)(33.77,18.52)
\psbezier(33.1,16.63)(31.97,16.73)(30,18.84)
\psbezier(28.02,20.96)(25.57,21.25)(21.82,19.82)
}
\psline[linewidth=0.2,linestyle=dashed,dash=1 1](21,128)(74.39,128)
\psline[linewidth=0.2,linecolor=red,linestyle=dashed,dash=1 1](38,133.77)(20.07,133.68)
\rput{19.52}(19.5,127.23){\psellipticarc[linewidth=0.2,linecolor=red,linestyle=dashed,dash=1 1](0,0)(6.61,6.34){71.22}{160.53}}
\psline[linewidth=0.2,linecolor=red,linestyle=dashed,dash=1 1](12.71,126.55)(13.04,28.04)
\psbezier[linewidth=0.2,linecolor=red,linestyle=dotted](13.73,23.31)(13.73,23.31)(13.73,23.31)(13.73,23.31)
\rput{7.46}(22,27.23){\psellipticarc[linewidth=0.2,linecolor=red,linestyle=dashed,dash=1 1](0,0)(8.67,-7.56){97.48}{185.97}}
\rput[bl](86,33){$r$}
\rput[bl](23.82,146.18){$\theta$}
\rput[bl](85,126){$\theta=p_\theta(\widetilde\alpha(0))+\frac{2n\epsilon}{3}$}
\rput[bl](77,58.52){$\widetilde\alpha$}
\pscustom[linewidth=0]{\psline{<-}(21.28,18.96)(17.71,14.32)
\psline(17.71,14.32)(12.36,12.23)
\psbezier{-}(12.36,12.23)(12.36,12.23)(12.36,12.23)
}
\psline[linewidth=0]{<-}(44.68,129.01)
(52.36,135.62)(58.79,136.51)
\rput[bl](59.98,135.32){$\widetilde\alpha(s_0)$}
\rput[bl](5,10){$\widetilde\alpha(0)$}
\psline[linewidth=0]{<-}(63.85,64.23)
(68.99,61.3)
(68.89,61.3)(76.11,59.55)
\pscustom[linewidth=0]{\psline{<-}(76,128)(84,128)
\psbezier{-}(84,128)(84,128)(84,128)
}
\psline[linewidth=0]{<-}(11.46,52.01)(5.57,49.37)
\rput[bl](2,46){$\widetilde\beta$}
\rput(64.32,123.12){$Q$}
\rput{0}(38.43,129.19){\psellipticarc[linewidth=0.1,linecolor=red,linestyle=dashed,dash=1 1](0,0)(5.71,-4.55){-85.53}{9.2}}
\end{pspicture}}
\caption{By concatenating $\widetilde\alpha$ with three circle arcs and two segment lines, we obtain a negatively orientated simple closed curve.}
\end{figure}
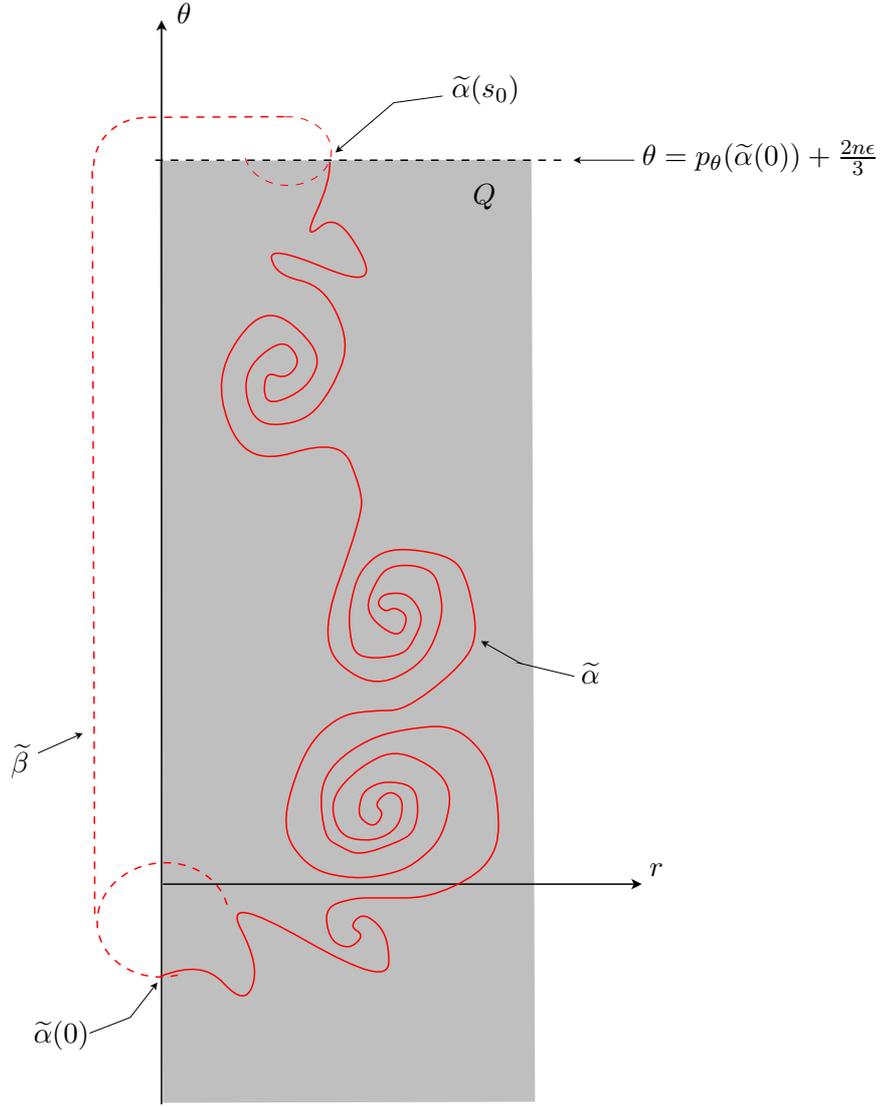

\begin{rem}
Note that the choice of $s_0$ is important: our proof of claim~\ref{c.2}  uses the fact that $s_0$ is \emph{the smallest} value of the parameter $s$ which satisfies the inequality $\widetilde u(s,n)-\widetilde u(0,n) \geq \frac{2n\epsilon}{3}$. This ensures that the curve $\widetilde\alpha$ is contained in the quarter of the  plane $Q$, and hence, prevents $\widetilde\alpha$ from ``spiraling" arround its extremity $\alpha(s_0)$.
\end{rem}

Now, using claim~2, it is not difficult to finish the proof of lemma~\ref{lemma.linking-implies-torsion}. Indeed, write
\begin{eqnarray*}
\mathrm{Torsion}_n(I,z(s_0),\xi) & = & \frac{1}{n}(\widetilde v(s_0,n)-\widetilde v(s_0,0)) \\
& = & \frac{1}{n}\Big(\mathop{\underbrace{(\widetilde v(s_0,n)-\widetilde v(0,n))}}_{A} + \mathop{\underbrace{(\widetilde v(0,n)-\widetilde v(0,0))}}_{B} + \mathop{\underbrace{(\widetilde v(0,0)-\widetilde v(s_0,0))}}_{C}\Big).
\end{eqnarray*}
Claim~2 states that quantity~A is bigger than $\frac{2n\epsilon}{3}-\frac{3}{4}$. Quantity $B$ is equal to $-n\mathrm{Torsion}_n(I,x,\xi)$, and we assumed that this quantity is bigger than $-\frac{n\epsilon}{3}$ (hypothesis~\eqref{e.hypothese-torsion-x}). Finally,  quantity~C is equal to $0$ since $v(s,0)=\xi$, for every~$s$. So we get
 $$\mathrm{Torsion}_n(I,z(s_0),\xi)\geq \frac{\epsilon}{3}-\frac{3}{4n}\geq \frac{\epsilon}{3}-\frac{1}{n},$$
 and lemma \ref {lemma.linking-implies-torsion} is proved.
\end{proof}


\section{Proof of theorem \ref{thm1}}  
\label{s.thm1}

In this section, we consider a  $C^1$-diffeomorphism $f$ of the open unit disc $\DD^2$. We assume that $f$ has compact support (\emph{i.e.} coincides with the identity outside a compact subset of $\DD^2$), and preserves the standard area two-form, which we denote by $\omega$. The goal of the section is to prove theorem~\ref{thm1}, \emph{i.e.} to prove that $f$ has an orbit with non-zero torsion. 

We will compute linking numbers and torsion using the natural trivialization of the tangent bundle of $\DD^2$, and an isotopy $I=(f_t)_{t\in [0,1]}$ joining the identity to $f$ in the space of $C^1$-diffeomorphisms of $\DD^2$ with compact supports. We have already noticed at the end of section~\ref{s.definition-torsion} that the torsion of an orbit of $f$ does not depend on the choice of an isotopy joining the identity to $f$. This will allow us to assume that $I$ is an isotopy in the set of area preserving diffeomorphisms of $\DD^2$ (the existence of such an isotopy follows from Moser's lemma). As usual, we extend the isotopy $I$ by setting $f_t:=f_{t-n} \circ f^{n}$ where $n=\lfloor t\rfloor$, for every $t\in \RR$. 

Let us first observe that the average linking number, with respect to $\omega$, of the orbits with a given point $x_0$ is well-defined:

\begin{pro}
\label{Birkh}
For $x_0 \in \DD^2$, the quantity $\mathrm{Linking}(I,x_0,x)$ is defined for almost every $x$ in $\DD^2$, 
the function $x\mapsto \mathrm{Linking}(I,x_0,x)$ is integrable, and 
$$\displaystyle\int_{\D^{2}} \mathrm{Linking}(I,x_0,x) d\omega(x)= \displaystyle\int_{\D^{2}} \mathrm{Linking}_1(I,x_0,x) d\omega(x).$$
\end{pro}

 \begin{proof}[Proof of proposition~\ref{Birkh}]
 We denote by $\AA$ the annulus obtained from $\D^2$ by blowing up $x_0$. The function $x \mapsto \mathrm{Linking}_1(I,x_0,x)$, which is defined on $\DD^2\setminus\{x_0\}$, extends continuously on $\AA$ (using the derivative of $f$ at $x_0$). This shows that $x \mapsto \mathrm{Linking}_1(I,x_0,x)$ is bounded on $\DD^2\setminus\{x_0\}$. Proposition~\ref{Birkh} then follows from Birkhoff's ergodic theorem.
  \end{proof}
   
Denote by $\mathrm{Fix}(f)$ be the set of the fixed points of $f$. Theorem~\ref{thm1} is an immediate consequence of the following proposition, together with proposition~\ref{prop.linking-implies-torsion}.

\begin{pro}\label{pro1}
If $f$ is not the identity, then there exists a point $x_0 \in Fix(f)$ such that 
$$\displaystyle\int_{\D^{2}} \mathrm{Linking}(I,x_0,x) d\omega(x) \not =0.$$
 \end{pro}

Proposition \ref{pro1} is actually a reformulation, in the context surfaces dynamics, of a symplectic geometry result due to C.~Viterbo (proposition~\ref{viterbo} below). For every $t\in [0,1]$, let $X_t$ be the vector field on $\D^2$ defined by $X_t(x)=\frac{d}{ds}_{|s=t} f_s(x)$. Then, for every $t\in [0,1]$, there is a unique normalised function $H_t:\D^2\to\RR$ which ``generates the vector field $X_t$". By such, we mean that $H_t$ is null outside a compact subset of $\DD^2$ and  satisfies $\omega(X_t,\cdot)=dH_t$. This allows to define the symplectic action (for the isotopy $I$) of a point $x\in\mathrm{Fix}(f)$:

\begin{defi}
Let $x$ be a point in $\mathrm{Fix}(f)$. The \emph{symplectic action} (for the isotopy $I$) of the fixed point $x$ is the quantity 
 $$A_I(x)=\displaystyle\int_{\gamma_x} \; \lambda -  \displaystyle\int_{0}^1 H_t(f_t(x)) \; dt,$$
 where $\gamma_x:[0,1]\to\DD^2$ is the loop defined by $\gamma_x(t)=f_t(x)$ and $\lambda$ is a primitive of the area form $\omega$.
\end{defi}

Viterbo's result reads as follows:

\begin{pro}[Viterbo; see proposition 4.2. of \cite{Viterbo}]
\label{viterbo}
If $f$ is not the identity, then 
$$\mathop{\sup}_{x\in\mathrm{Fix}(f)}\left\{A_I(x)\right\} \not= \mathop{\inf}_{x\in\mathrm{Fix}(f)}\left\{A_I(x)\right\}.$$
\end{pro}

Now, we will see that the hamiltonian action $A(x_0)$ of a point $x_0 \in\mathrm{Fix}(f)$ is nothing else that the average linking around $x_0$, that is $\mathrm{Linking}(I,x_0,\mu)$. The following lemma is well-known by experts in hamiltonian dynamics; the first author learned it from P. Le Calvez.

 \begin{lem}
 \label{lem4} 
 We have $$A_I(x_0)=\mathrm{Linking}(I,x_0,\omega).$$
  \end{lem}
  
\begin{proof}
The symplectic action $A_I(x_0)$ does not depend on the class of homotopy of the isotopy $I$ (see e.g. \cite{Viterbo}). 
If $x,y\in \DD^2$, it is easy to find an area preserving diffeomorphism with compact support $\tau_{x,y}$ such that $\tau_{x,y}(x)=y$, such that $\tau_{x,y}=\mathrm{Id}$ if $x=y$, and such that $\tau_{x,y}$ depends continuously on the couple $(x,y)$. Up to replacing $f_t$ by $\tau_{x_0,f_t(x_0)}\circ f_t$, we may therefore assume that the isotopy $I=(f_t)_{t \in [0,1]}$ fixes the point $x_0$, \emph{i.e.} satisfies $f_t(x_0)=x_0$ for all $t$. Under this assumption, the formula for the symplectic action of the point $x_0$ reads
$$A_I(x_0)= - \int_{0}^1 H_t(x_0) \; dt.$$
We are left to prove that $\mathrm{Linking}(I,x_0,\omega)$ is equal to the integral on the right-hand side above. For $x\in\DD^2\setminus\{x_0\}$, let $\theta(x)\in\RR/\ZZ$ be the argument of $\frac{x-x_0}{\|x-x_0\|}\in\SS^1$. Note that $d\theta$ is a closed $1$-form on $\D^2\setminus\{x_0\}$. By definition, for every $x\in\DD^2\setminus\{x_0\}$, the quantity $\mathrm{Linking}_1(I,x,x_0)$ is the variation of $\theta(f_t(x))$ when $t$ runs from $0$ to $1$. Since $X_t(x)=\frac{d}{ds}_{|s=t} f_s(x)$, this yields
$$\mathrm{Linking}_1(I,x,x_0)=\int_0^1 d\theta(X_t(f_t(x)) \;dt$$
So, if we denote by $D_\epsilon$  the disc of radius $\epsilon$ and centered at $x_0$, we have
\begin{eqnarray*}
\mathrm{Linking}(I,x_0,\omega)  
& = & \int_{\D^{2}\setminus \{x_0\}} \mathrm{Linking}_1(I,x,x_0) \; \omega(x)\\
& = &  \int_{\D^{2}\setminus \{x_0\}} \left( \int_0^1 d\theta\big(X_t\big(f_t(x)\big)\big) dt\right) \omega(x)  \\
& = & \int_0^1 \left(\int_{\D^2\setminus \{x_0\}} d\theta\big(X_t\big(f_t(x)\big)\big) \omega(x) \right) dt   \\
& = &  \int_0^1 \left(\int_{\D^2\setminus\{x_0\}} d\theta\big(X_t(x)\big) \; \omega(x) \right) dt\\
& = &  \int_0^1 \left(\lim_{\epsilon\to 0}\int_{\D^2\setminus D_\epsilon} d\theta(X_t) \; \omega \right) dt
\end{eqnarray*}
(the third equality uses the fact that $f_t$ preserves $\omega$ and fixes $x_0$). Now observe that
$$d\theta(X_t)\;\omega = d\theta \wedge \omega(X_t,\cdot) = -d(H_t \;d\theta)$$ 
(the first equality follows from the nullity of the 3-form $d\theta\wedge \omega$, whereas the second equality is a consequence of the nullity of the form $dd\theta$). Using this and Stokes' theorem, we obtain
\begin{eqnarray*}
\mathrm{Linking}(I,x_0,\omega)  & = &  - \int_0^1\left( \lim_{\epsilon \rightarrow 0} \displaystyle\int_{\DD^2\setminus D_\epsilon}   d(H_t \;d\theta)\right) dt \\
& = & - \int_0^1\left( \lim_{\epsilon \rightarrow 0} \displaystyle\int_{\partial D_\epsilon}  H_t\; d\theta\right) dt \\
& = & -\int_0^1 H_t(x_0) dt.
\end{eqnarray*}
This completes the proof of lemma~\ref{lem4}.
\end{proof}

\begin{proof}[Proof of proposition~\ref{pro1}]
If $f$ is not the identity, Viterbo's proposition~\ref{viterbo} shows that there is at least one point $x_0\in\mathrm{Fix}(f)$ such that the symplectic action $A_I(x_0)$ is different from zero, and lemma~\ref{lem4} shows that $A_I(x_0)$ is equal to the integral $\int_{\DD^2}\mathrm{Linking}(I,x_0,x)\;\omega(x)$.
\end{proof}

Proposition~\ref{pro1} can also be deduced from of Le Calvez's equivariant foliated version of the Brouwer plane translations theorem (\cite{LeCalvez}), together with a recent result of O. Jaulent, which allows  to apply Le Calvez's result in a situation where there are infinitely many fixed points (\cite{Jaulent}). See~\cite[proposition~2.1., item (5)]{BFLM} for more details. Observe that this alternative proof also works in the case where the area form $\omega$ is replaced by any $f$-invariant probability measure, whose support in not contained in $\mathrm{Fix}(f)$. 

\begin{proof}[Proof of theorem~\ref{thm1}]
Assume that $f$ is not the identity. Proposition~\ref{pro1} shows the existence of a point $x_0\in\mathrm{Fix}(f)$ and a point $x\in\DD^2\setminus\{x_0\}$, such that the linking number $\mathrm{Linking}(I,x_0,x)$ is not zero. According to proposition \ref{prop.linking-implies-torsion}, this implies the existence of a point $z\in \DD^2$ such that the torsion 
$\mathrm{Torsion}(I,z)$ is not null.
\end{proof}

  \section{Proof of theorem \ref{thm2}}
  \label{s.thm2}

The purpose of this section is to prove theorem~\ref{thm2}. Before starting the proof strictly speaking, we need to define a notion of \emph{linking number} for a pair of curves.

\begin{defi}\label{defi2}
Let $\alpha,\beta:\RR \mapsto \R^2$  be two curves such that $\alpha(t)\neq\beta(t)$ for all $t\in\R$. Consider the function $v:\R\to\SS^1$ given by $v(t)= \frac{\beta(t)-\alpha(t)}{\|\beta(t)-\alpha(t)\|}.$ Choose a continuous lift $\widetilde{v}:\R \to \R$ of $v$. The \emph{linking number} of the curves $\alpha$ and $\beta$ is the quantity 
$$\mathrm{Linking}(\alpha,\beta):=\mathop{\lim}_{t\rightarrow+\infty}\frac{1}{t}\Big(\widetilde{v}(t)-\widetilde{v}(0)\Big),$$ 
provided that the limit exists. 
\end{defi}

\begin{rem}
\label{r.linking-curve-orbits}
If the curves $\alpha,\beta$ are defined by $\alpha(t)=f_t(x)$ and $\beta(t)=f_t(y)$ for some isotopy $I=(f_t)_{t\in [0,1]}$ on a surface $S$ and for some points $x,y\in S$, then $\mathrm{Linking}(\alpha,\beta)=\mathrm{Linking}(I,x,y)$.
\end{rem}

The following technical lemma will be used twice in the proof of  theorem~\ref{thm2}:

\begin{lem}
\label{lem3}
Consider two curves $\alpha:\R \to \R^2$ and $\beta:\R\to\R^2$, and assume that there is a positive constant $d$ such that $\|\beta(t)-\alpha(t)\|\geq d$ for all $t \in \R$. Consider two curves $\alpha':\R \to \R^2$ and $\beta':\R\to\R^2$ such that $\|\alpha(t)-\alpha'(t)\|\leq \frac{d}{2}$ and $\|\beta(t)-\beta'(t)\|\leq \frac{d}{2}$ for all $t\in\RR$. Then, 
 $$\mathrm{Linking}(\alpha,\beta)=\mathrm{Linking}(\alpha',\beta').$$
\end{lem}

\begin{proof}
Since $\left\|\big(\alpha(t)-\alpha'(t)\big)+\big(\beta(t)-\beta'(t)\big)\right\|\leq  \|\beta(t)-\alpha(t)\|$, the angle between the vectors $\beta(t)-\alpha(t)$ and $\beta'(t)-\alpha'(t)$ is less than $\frac{1}{4}$ turn. So, $\frac{1}{t}\Big|\mathrm{Linking}_t(\alpha,\beta)-\mathrm{Linking}_t(\alpha',\beta')\Big| \leq \frac{1}{4t}.$
Letting $t\rightarrow+\infty$, we obtain $\mathrm{Linking}(\alpha,\beta)=\mathrm{Linking}(\alpha',\beta').$ 
\end{proof}

 \bigskip 
  
We are now ready to begin the proof of theorem~\ref{thm2}. We consider a diffeomorphism $f$ of the torus $\T^2= \R^2/\Z^2$ isotopic to the identity and an isotopy $I=(f_t)_{t\in [0,1]}$ joining the $\mathrm{Id}_{\T^2}$ to $f$ in $\mathrm{Diff}^1(\TT^2)$. This isotopy $I$ can be lifted to an isotopy $\widetilde I=(\widetilde f_t)_{t\in [0,1]}$  joining $\mathrm{Id}_{\R^2}$ to a lift $\widetilde f$ of $f$ in $\mathrm{Diff}^1(\RR^2)$. As usual, for $t\in\R$, we set  $f_t=f_{t-n} \circ f^{n}$ and $\widetilde f_t=\widetilde f_{t-n} \circ f^{n}$ where $n=\lfloor t\rfloor$.  We assume that the rotation set of $\widetilde f$ has non-empty interior. Recall that this hypothesis depends only on $f$, and not on the choice of the lift $\widetilde f$ (two different lifts have the same rotation set up to a translation). We trivialize the tangent bundle of $\TT^2$, using the trivialization induced by the affine structure of $\RR^2$. To prove theorem~\ref{thm2}, we have to find a point $\bar z \in \TT^2$ such that the torsion $\mathrm{Torsion}(I,\bar z)$ is non-zero. We recall that this is equivalent to finding a point $z \in \RR^2$ such that the torsion $\mathrm{Torsion}(\widetilde I, z)$ is non-zero (see remark \ref{rem1}).
 
Let $(p/q,p'/q)$ be a vector with rational coordinates in the interior of the rotation set of $\widetilde f$. Let $g:=f^q$ and $\widetilde g := \widetilde f^q-(p,p')$.  It is easy to check that $\rho (\widetilde g) = q \rho (\widetilde f)-(p,p')$; in particular, $(0,0)$ is in the interior of $\rho (\widetilde g)$. Moreover, it is easy to check that $\mathrm{Torsion}(g,z)=q \mathrm{Torsion}(f,z)$, for every $z\in\RR^2$; in particular, $f$ has an orbit with non-zero torsion if and only if it is the case for $g$. Therefore, up to replacing $f$ by $g$, we may ---~and we will~--- assume that $(0,0)$ is in the interior of the rotation set of $\widetilde f$.

We will use the following lemma, which is due to J.~Franks:

 \begin{lem}[Franks, \cite{Franks}]
 \label{lemFranks}
Let $g$ be a homeomorphism of the torus $\TT^2$ that is isotopic to the identity, and $\widetilde g$ be a lift of $g$ to $\RR^2$. Let $(p/q,p'/q)$ be a vector with rational coordinates in the interior of the rotation set of $\widetilde g$. Then, there exits $z \in \RR^2$ such that $\widetilde g^q (z)=z+(p,p')$. In other words, each vector with rational coordinates $(p/q,p'/q)$ in the interior of the rotation set of $\widetilde g$ is realized as a rotation vector of a periodic point of $g$ of period $q$.
 \end{lem}

Since the rotation set of $\widetilde f$ has non-empty interior, we can find three affinely independent vectors $u_1,u_2,u_3$ with rational coordinates in the interior of the rotation set of $\widetilde f$. According to lemma \ref{lemFranks}, there exist three periodic orbits of $f$ with respective rotation vectors $u_1,u_2,u_3$. Let $E$ be the finite subset of $\TT^2$ made of the points of these three orbits. Llibre and MacKay \cite{Llibre-MacKay:91} have proved that $f$ is isotopic, relatively to $E$, to a pseudo-Anosov homeomorphism with marked points $\phi$. According to a well-known result of M. Handel (\cite{Han:85}), this implies that $\phi$ is a topological factor of $f$. More precisely, there exists a continuous surjection $h:\T^2\to \T^2$, which fixes each point of $E$, which is homotopic to the identity (through a homotopy which fixes each point of $E$), such that $h\circ f=\phi\circ h$. As $h$ is homotopic to the identity, it has a lift $\widetilde h:\R^2\to\R^2$ commuting to the action of $\Z^2$. Then the homeomorphism $\phi$ has a unique lift $\widetilde \phi:\R^2\to \R^2$ such that $\widetilde h\circ \widetilde f=\widetilde \phi\circ \widetilde h$. Now, note that, considered as homeomorphism of $\TT^2$, the homeomorphism $\phi$ is isotopic to the identity (since it is isotopic to $f$). So we may consider an isotopy $J=(\phi_t)_{t \in [0,1]}$ joining $Id_{\T^2}$ to $\phi$ in $\mathrm{Homeo}(\T^2)$, which lifts to 
an isotopy $\widetilde J=(\widetilde\phi_t)_{t \in [0,1]}$ joining $Id_{\R^2}$ to $\widetilde\phi$  in $\mathrm{Homeo}(\R^2)$. As usual, for $t \in \RR$, we set $\phi_t=\phi_{t-n} \circ \phi^{n}$ and $\widetilde\phi_t=\widetilde\phi_{t-n} \circ \widetilde\phi^{n}$ where $n=\lfloor t\rfloor$. For each $t \in \RR$, the homeomorphism $\widetilde\phi_t$ commutes to the action of $\ZZ^2$.

\begin{rem}
Since $h$ is not one-to-one, it is not possible in general to find an isotopy $(\phi_t)_{t \in [0,1]}$ joining $Id_{\T^2}$ to $\phi$ such that $h \circ f_t = \phi_t \circ h$, for every $t$.
\end{rem} 

Since  $\widetilde h$ commutes to the action of $\Z^2$, and since $\R^2/\Z^2$ is compact, $\widetilde h$ is at a finite uniform distance from the identity; we will denote this distance by $d_1$:
 $$d_1:=\left\|\widetilde h-Id_{\R^2}\right\|_{\infty}=\sup_{z \in \R^2}  \mathrm{dist}\left(\widetilde h(z),z\right)<+\infty$$      
 (we recall that $\mathrm{dist}$ denote the canonical euclidean distance on the affine space $\RR^2$). 
  
 \begin{lem}
 \label{lem5}
The homeomorphisms $\widetilde f$ and $\widetilde \phi$ have the same rotation set. In particular, $(0,0)$ is in the interior of the rotation set of $\widetilde\phi$. 
 \end{lem}

\begin{proof}
The proof uses from the conjugacy relation $\widetilde\phi\circ \widetilde h=\widetilde h\circ \widetilde f$, and the fact that $\widetilde h$ is at a finite uniform distance from the identity. For every point $x\in\R^2$ and every positive integer $n$, 
\begin{eqnarray*}
\left|\rho_n\left(\widetilde\phi,\widetilde h(x)\right)-\rho_n\left(\widetilde f,x\right)\right| & = & \left |\frac{1}{n}\left(\widetilde \phi^n\left(\widetilde h(x)\right)-\widetilde h  (x)\right)-\frac{1}{n}\left(\widetilde f^n(x)-x\right) \right|\\
& = & \frac{1}{n}\left|\left(\widetilde h\left(\widetilde f^n  (x)\right)-\widetilde f^n  (x)\right)-\left(\widetilde h(x)-x\right) \right|~\leq~\frac{2d_1}{n}.
\end{eqnarray*}
Letting $x$ range over $\RR^2$, we obtain that the Hausdorff distance between the sets $\rho_n(\widetilde\phi)$ and $\rho_n(\widetilde f)$ is less than $\frac{2d_1}{n}$. The lemma follows.
\end{proof}

Theorem~\ref{thm2} will appear as a consequence of the following lemma, together with proposition~\ref{prop.linking-implies-torsion} and lemma~\ref{lem3}:
 
\begin{lem}
\label{lem6}
Let $d$ be a positive constant. There exist two points $x,y \in \R^2$ such that
 \begin{enumerate}
 \item $x$ and $y$ are periodic points for $\widetilde \phi$,
 \item for every $t \in \R$, the distance $\mathrm{dist}(\widetilde\phi_t(x),\widetilde\phi_t(y))$ is bigger than $d$,
 \item the linking number $\mathrm{Linking}(\widetilde J,x,y)$ is not null (observe that this linking number is well-defined since the orbits of  $x$ and $y$ are periodic for $\widetilde\phi$).
 \end{enumerate}
 \end{lem}
 
 Let us postpone the proof of lemma~\ref{lem6}, and complete the proof of theorem~\ref{thm2} assuming that this lemma is true.
 
 \begin{proof}[Proof of theorem~\ref{thm2}, assuming lemma~\ref{lem6}]
 Since $\widetilde \phi_t$ commutes to the action of $\Z^2$ and depends continuously on $t$, we can find a constant $d_2$ with the following property:
 $$\mbox{if } \mathrm{dist}\left(x,y\right)\leq d_1, \mbox{ then }  \sup_{t\in [0,1]}(\mathrm{dist}\left(\widetilde\phi_t(x),\widetilde \phi_t(y)\right)\leq d_2 .$$
 Set $d:=2d_2$, and consider the points $x,y$ provided by lemma \ref{lem6}. Choose a point $x' \in \widetilde h^{-1}(\{x\})$ and a point $y' \in \widetilde h^{-1}(\{y\})$. Let $t \in \R$ and $\lfloor t\rfloor=n$. Since $\widetilde h\circ \widetilde f^n=\widetilde \phi^n \circ \widetilde h$ and $\widetilde h(x') = x$, one has
 $$\mathrm{dist}\left(\widetilde\phi^n(x),\widetilde f^n(x')\right) = 
\mathrm{dist}\left(\widetilde h \left(\widetilde f^n(x')\right),\widetilde f^n(x')\right)  \leq d_1.$$
Hence, 
$$\mathrm{dist}\left(\widetilde\phi_t(x),\widetilde f_t(x')\right) = \mathrm{dist}\left(\widetilde \phi_{t-n} \left(\widetilde \phi^n(x)\right), \widetilde f_{t-n}\left(\widetilde f^n(x')\right)\right)  \leq d_2 = \frac{d}{2}.$$ 
Similar arguments yield 
$$\mathrm{dist}\left(\widetilde\phi_t(y),\widetilde f_t(y')\right) \leq \frac{d}{2}.$$
Let $\alpha,\beta,\alpha',\beta':\R\to\R^2$ be the curves defined by $\alpha(t)=\widetilde\phi_t(x)$, $\beta(t)=\widetilde\phi_t(y)$,  $\alpha'(t)=\widetilde f^t(x')$ and $\beta'(t)=\widetilde f^t(y')$. The inequalities above yield
  $$\mathrm{dist}\left(\alpha(t),\alpha'(t)\right)\leq  \frac{d}{2}\;\; \mathrm{and}\;\;\mbox{dist}\left(\beta(t),\beta'(t)\right)\leq \frac{d}{2}\;\; \mbox{for all}\;\;t,$$ 
  and property~(2) of lemma \ref{lem6} yields
$$\mathrm{dist}(\alpha(t),\beta(t))>d,\;\;\mbox{for all}\;\;t\in\R.$$ 
Moreover, according to the item~(3) of lemma \ref{lem6}, the quantity $\mathrm{Linking}(\alpha, \beta)$ is non-zero. So lemma \ref{lem3} implies that $\mathrm{Linking}(\alpha', \beta')$ is also non-zero. Now observe that $\mathrm{Linking}(\alpha', \beta')=\mathrm{Linking}(\widetilde I, x, y)$. Hence, we have found two points $x,y\in\RR^2$ such that $\mathrm{Linking}(\widetilde I, x, y)$ is non-zero. According to proposition~\ref{prop.linking-implies-torsion}, this implies the existence of a point $z\in\R^2$ such that $\mathrm{Torsion}(\widetilde I,z)$ is non-zero. If $\bar z$ is the projection of $z$ in $\TT^2$, then we have $\mathrm{Torsion}(I,\bar z)=\mathrm{Torsion}(\widetilde I,z)\neq 0$. This completes the proof of theorem~\ref{thm2}.
\end{proof}

Now, we are left to prove lemma~\ref{lem6}. The main tool of the proof will be a Markov partition for the diffeomorphism $\widetilde\phi$. Let us recall some basic facts about Markov partitions.

Let $g$ be homeomorphism of a (non-necessarly compact) surface $S$. A \emph{rectangle} in $S$ is a topological embedding of $[0,1]^2$ in $S$. A \emph{Markov partition} for $g$ is a covering of $S$ by a locally finite collection of rectangles with pairwise disjoint interiors, such that, for every $i,j\in I$ $g(R_i)$ intersects $R_j$ in a certain way (we will not need the precise definition, but only some properties that we will be stated later; see e.g.~\cite{FLP} for a precise definition). Given a Markov partition $\cM=\{R_i\}_{i\in I}$, we call a \emph{$g$-chain of rectangles of $\cM$} a finite sequence $c=(R_{i_1},\dots,R_{i_n})$ of rectangles in $\cM$ such that $g(R_{i_k})$ intersects $R_{i_{k+1}}$ for $i=1,\dots,n-1$. The \emph{length} of a $g$-chain $c=(R_{i_1},\dots,R_{i_n})$ is $n-1$. A close $g$-chain of rectangles of $\cM$ is of course a $g$-chain $c=(R_{i_1},\dots,R_{i_n})$ such that $R_{i_n}=R_{i_1}$. If $x\in S$ is a periodic point of period $p$ for $g$, the \emph{$\cM$-itinerary} of $g$ is the closed $g$-chain of rectangles $c=(R_{i_0},\dots,R_{i_p}=R_{i_0})$ such that $g^k(x)\in R_{i_k}$ for $k=0\dots p$. We will use the following properties of Markov partitions:

\medskip

\begin{facts}[see for example~\cite{FLP}]
\label{f.Markov}~

\begin{enumerate}
\item  If $\cM=\{R_i\}_{i\in I}$ is a Markov partition for a surface homeomorphism $g$, then every closed $g$-chain of rectangles of~$\cM$ of length $p$ is the itinerary of a periodic point of period $p$.
\item Every pseudo-Anosov homeomorphism admits a (finite) Markov partition.
\item  Let $g$ be a homeomorphism of a surface $S$, and $\widetilde g$ be a lift of $g$ to the universal covering~$\widetilde S$ of~$S$. Let $\cM$ be a Markov partition for $g$, and $\widetilde\cM$ be the lift of $\cM$ in $\widetilde S$ (\emph{i.e.} the collection of all the lifts of the rectangles of $\cM$). Then $\widetilde\cM$ is a Markov partition for $\widetilde g$. 
\end{enumerate}
\end{facts}

We are now ready to begin the proof of lemma~\ref{lem6}.

\begin{proof}[Proof of lemma~\ref{lem6}]
According to the second item of facts~\ref{f.Markov}, the pseudo-Anosov homeomorphism $\phi$ admits a finite Markov partition $\cM$. We denote by $\widetilde \cM$ the lift of this Markov partition in $\RR^2$ (\emph{i.e.} the collection of all the lifts of the rectangles of $\cM$). According to the facts stated above, $\widetilde\cM$ is a Markov partition for $\widetilde\phi$. Note that, if $\widetilde R$ is a rectangle in $\widetilde \cM$, and $(p,p')$ is a vector in $\ZZ^2$, then $\widetilde R+(p,p')$ is also a rectangle of $\widetilde\cM$.  

To prove the lemma \ref{lem6}, we will construct a closed $\widetilde\phi$-chain $\Gamma$ of rectangles of $\widetilde\cM$ with the following properties: one can find a fundamental domain $\Delta$ for the action of $\Z^2$ on $\RR^2$, such that ``$\Gamma$ makes one turn around $\Delta$" and ``$\Gamma$ stays far away from $\Delta$". 
 The closed chain $\Gamma$  will be the itinerary of a periodic point $y$ of $\widetilde\phi$. Franks'lemma \ref{lemFranks} will provide us with a fixed point $x$ of $\widetilde\phi$ in $\Delta$. The property `$\Gamma$ makes one turn around $\Delta$" will imply that the orbits of $x$ and $y$ will have non-zero linking number (item~(3) of lemma~\ref{lem6}). The property ``$\Gamma$ stays far away from $\Delta$" will imply that the orbits of $x$ and $y$ will stay far from each other ((item~(2) of lemma~\ref{lem6}). Let us start  the construction of the closed $\widetilde\phi$-chain $\Gamma$:

\setcounter{claim}{0}
\begin{claim} 
\label{Affirm3}
There exists a constant $K\in \N$ such that for every couple of rectangles $(\widetilde S,\widetilde T)$ in $\widetilde \cM$, there exists a vector $(p,p')\in\ZZ^2$ and a finite $\widetilde \phi$-chain of rectangles of $\widetilde\cM$ of length smaller than $K$ going from $\widetilde S$ to $\widetilde T+(p,p')$.
\end{claim}

\begin{proof}
We recall that a pseudo-Anosov homeomorphism is always transitive (see e.g.~\cite{FLP}). The topological transitivity of $\phi$ implies that, for every couple of rectangles $S,T$ of $\cM$, there exists an integer $k_{ST}$ such that $\phi^{k_{ST}}(S)$ intersects $T$. Hence, for any lift $\widetilde S$ of $S$ and any lift $\widetilde T$ of $T$, there is a vector $(p,p')\in\ZZ^2$ such that $\widetilde\phi^{k_{ST}}(\widetilde S)$ intersects $\widetilde T+(p,p')$. In particular, there is a finite $\widetilde \phi$-chain of rectangles of $\widetilde\cM$ of length smaller than $k_{ST}$ joining the rectangle $\widetilde S$ to the rectangle $\widetilde T+(p,p')$. Since $\cM$ has a finite number of rectangles ,  the integers $k_{ST}$ are uniformly bounded. This completes the proof of the claim.
\end{proof}

We pick a rectangle $\widetilde R$ of $\widetilde\cM$. 

\begin{claim} 
\label{Affirm4}
There exist three vectors $(p_1,p_1'),\; (p_2,p_2'),\;(p_3,p_3')$ in $\ZZ^2$ such that:
\begin{enumerate}
\item $(0,0)$ is in the interior of the convex hull of the vectors $(p_1,p_1'),\; (p_2,p_2'),\;(p_3,p_3')$,
\item for $i \in \{1,2,3\}$, there exists a finite $\widetilde\phi$-chain of rectangles of  $\widetilde\cM$, denoted $c_i$, joining $\widetilde R$ to a translate $\widetilde R+(p_i,p_i')$ of $\widetilde R$. We denote $k_i$ the length of the chain $c_i$.
\end{enumerate}
\end{claim}

\begin{proof} 
Let $K$ be the constant given by claim~\ref{Affirm3}. Since $\widetilde\phi$ is at finite uniform distance from the identity, there exists a constant $L$ with the following property:  if $\widetilde S$ and $\widetilde T$ are two rectangles of $\widetilde\cM$ such that there is a $\widetilde \phi$-chain of rectangles of length smaller than $K$ going from  $\widetilde S$ to  $\widetilde T$, then $\mathrm{dist}(x,y)\leq L$ for every $x\in\widetilde S$ and every $y\in\widetilde T$. Since $(0,0)$ is in the interior of the rotation set of $\widetilde\phi$, we can find three vectors $v_1,v_2,v_3$ in the rotation set of $\widetilde\phi$ such that $(0,0)$ is in the interior of the convex hull of $v_1,v_2,v_3.$ Then, we can find a constant $\eta > 0$ such that if $w_1,w_2,w_3$ are vectors in $\R^2$ such that $\|w_i-v_i\| < \eta$ for $i=1,2,3$, then $(0,0)$ is in the interior of the convex hull of $w_1,w_2,w_3$. 

Let $N$ be an integer such that $\frac{L}{N}<\frac{\eta}{3}$. For $i=1,2,3,$ since $v_i$ is in the rotation set of $\widetilde\phi$, we can find a point $z_i \in \R^2$ and an integer $n_i\geq N$ such that 
\begin{equation}
\label{e.presque-v_i}
\left\|\frac{1}{n_i}\left(\widetilde{\phi}^{n_i}(z_i)-z_i\right)-v_i\right\| < \frac{\eta}{3}.
\end{equation}
Let $\widetilde S_i$ and $\widetilde T_i$ be the rectangles of $\widetilde\cM$ containing respectively the points $z_i$ and $\widetilde\phi^{n_i}(z_i)$. Now we construct the $\widetilde\phi$-chain of rectangles $c_i$. 
\begin{itemize}
\item[--] Claim~1 provides us with a vector $(r_i,r_i')$ in $\ZZ^2$ and a $\widetilde\phi$-chain of rectangles of $\widetilde\cM$ of length less than $K$ going from $\widetilde R$ to $\widetilde S_i+(r_i,r_i')$. 
\item[--] By definition of the rectangles $\widetilde S_i$ and $\widetilde T_i$, and since $\widetilde\phi$ commutes with the action of $\ZZ^2$, the image under $\widetilde\phi^{n_i}$ of the rectangle $\widetilde S_i+(r_i,r_i')$ intersects the rectangle $\widetilde T_i+(r_i,r_i')$. In particular, there is a $\widetilde\phi$-chain of rectangles of $\widetilde\cM$ going from $\widetilde S_i+(r_i,r_i')$ to $\widetilde T_i+(r_i,r_i')$. 
\item[--] Claim~1 provides us with a vector $(p_i,p_i')$ in $\ZZ^2$ and a $\widetilde\phi$-chain of rectangles of $\widetilde\cM$ of length less than $K$ going from $\widetilde T_i+(r_i,r_i')$ to $\widetilde R+(p_i,p_i')$. 
\end{itemize}
Concatenating these three $\widetilde\phi$-chains of rectangles of $\widetilde\cM$, we obtain a $\widetilde\phi$-chain of rectangles of $\widetilde\cM$ going from $\widetilde R$ to $\widetilde R+(p_i,p_i')$, which we denote by $c_i$.

We will now prove that $\|\frac{1}{n_i}(p_i,p_i')-v_i\| < \eta$. Fix a point $z$ in $\widetilde R$. Since there exists a $\widetilde\phi$-chain of rectangles of $\widetilde\cM$ of length less than $K$ going from $\widetilde R$ to $\widetilde S_i+(r_i,r_i')$, we have $\|z-(z_i+(r_i,r_i'))\|\leq L.$ The same arguments yield $\|\widetilde\phi^{n_i}(z_i+(r_i,r_i'))-(z+(p_i,p_i'))\|\leq L.$ So, we get
\begin{eqnarray*}
&& \Big\|\frac{1}{n_i}(p_i,p_i')-v_i\Big\| = \Big\|\frac{1}{n_i}\left(\left(z+(p_i,p_i')\right)-z\right)-v_i\Big\|\\
& \leq &\frac{1}{n_i} \left\| (z+(p_i,p_i'))-\widetilde\phi^{n_i}(z_i+(r_i,r_i')) \right\| + \Big\| \frac{1}{n_i} \left(\widetilde{\phi}^{n_i}(z_i+(r_i,r_i'))-(z_i+(r_i,r_i'))\right)-v_i\Big\| \\
&&+ \frac{1}{n_i}\Big\|\left(z_i+(r_i,r_i')\right)-z\Big\|\\
& \leq & \frac{L}{n_i}+\frac{\eta}{3}+\frac{L}{n_i} \;\leq\;\eta. 
\end{eqnarray*}

By definition of $\eta$, these inequalities imply that $(0,0)$ is in the interior of the convex hull of the vectors $\frac{1}{n_1}(p_1,p_1'),\; \frac{1}{n_2}(p_2,p_2'),\; \frac{1}{n_3}(p_3,p_3')$. Therefore $(0,0)$ is also in the interior of the convex hull of the vectors $(p_1,p_1'),\; (p_2,p_2'),\;(p_3,p_3')$.
\end{proof}

We will now construct closed $\widetilde\phi$-chains of rectangles of $\widetilde\cM$.  If $c=(\widetilde R_1,\dots,\widetilde R_n)$ is a $\widetilde\phi$-chain of rectangles of $\widetilde\cM$, and $(p,p')$ is a vector in $\ZZ^2$, we denote by $c+(p,p')$ the sequence of rectangles $(\widetilde R_1+(p,p'),\dots,\widetilde R_n+(p,p'))$. Observe that $c+(p,p')$ is a $\widetilde\phi$-chain of rectangles of $\widetilde\cM$, since $\widetilde\phi$ commutes with the action of $\ZZ^2$. If $c$ and $c'$ are two $\widetilde\phi$-chains of rectangles of $\widetilde\cM$ such that the last rectangles of $c$ is equal to the first rectangle of $c'$, we denote by $c\vee c'$ the $\widetilde\phi$-chains of rectangles obtained by concataneting $c$ and $c'$. 

Since $(p_1,p_1')$, $(p_2,p_2')$, $(p_3,p_3')$ are vectors with integral coordinates, and since $(0,0)$ is in the interior of the convex hull of these vectors, there exists three positive integers $\ell_1,\ell_2,\ell_3$ such that
$$\ell_1(p_1,p_1')+\ell_2(p_2,p_2')+\ell_3(p_3,p_3')=(0,0).$$
For $i=1,2,3$, and $n\in \NN\setminus\{0\}$, we consider the $\widetilde\phi$-chain of rectangles of $\widetilde\cM$  defined as follows:
$$\Gamma_{i,n}=c_i \vee \Big(c_i+(p_i,p_i')\Big) \vee \Big(c_i+2(p_i,p_i')\Big) \vee\dots\vee \Big(c_i+(n\ell_i-1)(p_i,p_i')\Big).$$
This is a chain of length $n\ell_ik_i$, joining the rectangle $\widetilde R$ to the rectangle $\widetilde R+n\ell_i(p_i,p_i')$. Now, for $n\in\NN\setminus\{0\}$, we consider the $\widetilde\phi$-chain of rectangles of $\widetilde\cM$ defined as follows:
$$\Gamma_n:=\Gamma_{1,n}\vee \Big(\Gamma_{2,n}+n\ell_1(p_1,p_1')\Big)\vee\Big(\Gamma_{3,n}+n\ell_1(p_1,p_1')+n\ell_2(p_2,p_2')\Big).$$
This is a closed $\widetilde\phi$-chain of rectangles of $\widetilde\cM$, starting and ending at $\widetilde R$. We denote by $p_n=n(\ell_1k_1+\ell_2k_2+\ell_3k_3)$ the length of this $\widetilde\phi$-chain $\Gamma_n$.

For every $n\in\NN\setminus\{0\}$, since $\Gamma_n$ is a closed $\widetilde\phi$-chain of rectangles of $\widetilde\cM$ of length $p_n$, and since $\widetilde\cM$ is a Markov partition for $\widetilde\phi$, there exists a point $y_n\in\widetilde R$ which is periodic of period $p_n$ for $\widetilde\phi$, and such that the itinerary of the orbit of $y_n$ is precisely $\Gamma_n$. We will prove that the orbit of $y_n$ is very close from an affine triangle. We pick a point $\widetilde y$ in the rectangle $\widetilde R$. For every $n\in\NN\setminus\{0\}$, we consider the triangle 
$$T_n:=\mathrm{Conv}\left(y , y+n\ell_1(p_1,p_1') , y+n\left(\ell_1(p_1,p_1')+\ell_2(p_2,p_2'\right)\right).$$
Recall that $p_n=n(\ell_1k_1+\ell_2k_2+\ell_3k_3)$. We consider the parametrisation $\tau_n:[0,p_n]\to\partial T_n$ of the boundary of $T_n$ defined by the following properties:
\begin{itemize}
\item[--] $\tau_n$, is affine on the interval $[0,n\ell_1k_1]$, $[n\ell_1k_1,n(\ell_1k_1+\ell_2k_2)]$ and $[n(\ell_1k_1+\ell_2k_2),p_n]$, 
\item[--] $\tau_n(0)=y$,  $\tau_n(n\ell_1k_1)=y+n\ell_1(p_1,p_1')$, $\tau_n\left(n(\ell_1k_1+\ell_2k_2)\right)=y+n\left(\ell_1(p_1,p_1')+\ell_2(p_2,p_2')\right)$ and $\tau_n(p_n)=y$. 
\end{itemize}
The trajectory of $t\mapsto \widetilde\phi_t(y_n)$ and the triangle $T_n$ are depicted on figure~\ref{f.triangle}. The following claim states that the distance from the trajectory of $t\mapsto \widetilde\phi_t(y_n)$ to the boundary of the triangle $T_n$ is bounded independently of $n$.

\begin{claim} 
\label{claim-3}
There exists $D$ such that $\mathrm{dist}\left(\widetilde\phi_t(y_n) ,\tau_n(t) \right)\leq D$ for every $t\in [0,p_n]$ and every $n\in\NN$.
\end{claim}

\begin{proof}
Let us first observe that, for $t\in [0,n\ell_1k_1]$, one has $\tau_n(t)=y+\frac{t}{k_1}(p_1,p_1')$. Now consider the constant
$$D_1:=\sup_{z\in\widetilde R\,,\,t\in [0,k_1]} \mathrm{dist}\left(\widetilde\phi_t(z) \;,\; \tau_n(t)\right)=\sup_{z\in\widetilde R\,,\,t\in [0,k_1]} \mathrm{dist}\left(\widetilde\phi_t(z) \;,\; y+\frac{t}{k_1}(p_1,p_1')\right).$$
For every $n\in\NN\setminus\{0\}$ and every $j\in \{0,\dots,n\ell_1-1\}$, the point $\widetilde\phi_{jk_1}(y_n)-j(p_1,p_1')$ is in the rectangle $\widetilde R$. So, for every $n\in\NN\setminus\{0\}$, every $j\in \{0,\dots,n\ell_1-1\}$ and every $t\in [0,k_1]$, one has 
 $$\mathrm{dist}\left(\widetilde\phi_t\left(\widetilde\phi_{jk_1}(y_n)-j(p_1,p_1')\right) \;,\; y+\frac{t}{k_1}(p_1,p_1')\right) \leq D_1,$$
or equivalently
 $$\mathrm{dist}\left(\widetilde\phi_{jk_1+t}(y_n) \;,\; \tau_n(jk_1+t)\right) = \mathrm{dist}\left(\widetilde\phi_{jk_1+t}(y_n) \;,\; y+\left(j+\frac{t}{k_1}\right)(p_1,p_1')\right) \leq D_1.$$
 This shows that the distance between the points $\widetilde\phi_t(y_n)$ and $\tau_n(t)$ is bounded from above by $D_1$, for every $n\in\NN\setminus\{0\}$ and for every $t\in [0,n\ell_1k_1]$. Similar arguments show the existence of some constants $D_2$ and $D_3$, such that the distance between the points $\widetilde\phi_t(y_n)$ and $\tau_n(t)$ is bounded from above by $D_2$ and $D_3$, respectively for $t\in [n\ell_1k_1,n\ell_1k_1+n\ell_2k_2]$ and for $t\in [n\ell_1k_1+n\ell_2k_2,p_n]$. This completes the proof of claim~\ref{claim-3}.
\end{proof}

\begin{claim} 
\label{claim-4}
For $n$ large enough, there exists a fixed point $x_n$ of $\widetilde\phi$ such that, for every $t\in\RR$,
$$\phi_t(x_n)\in T_n\quad\mbox{and}\quad\mathrm{dist}\left(\phi_t(x_n)\;,\;\partial T_n\right) > D+d.$$
\end{claim}

\begin{proof}
Consider the real constant $d_3$ defined as follow:
$$d_3:=\sup_{t\in [0,1]} \left\|\widetilde\phi_t-\mathrm{Id} \right\|_\infty  = \sup_{z\in\RR^2}\sup_{t\inÊ[0,1]} \mathrm{dist}_{\RR^2}\left(\widetilde\phi_t(z),z\right).$$
The triangle $T_1$ has non-empty interior, and the triangle $T_n$ is the image of $T_1$ under an homothecy of ratio $n$. It follows that there exists $n_0\in\NN$ such that, for $n\geq n_0$, one can find a fundamental $\Delta_n$ for the action $\ZZ^2$ on $\RR^2$ such that $\Delta_n\subset T_n$ and $\mathrm{dist}(\partial\Delta_n,\partial T_n)\geq D+d+d_3$. Now, since $(0,0)$ is in the interior of the rotation set of $\widetilde\phi$, lemma~\ref{lemFranks} implies that $\widetilde\phi$ has a fixed point $x$.  For $n\geq n_0$, let $x_n$ be an element of $x+\ZZ^2$ which is in the fundamental domain $\Delta_n$. Note that $x_n$ is also a fixed point of $\widetilde\phi$, since $\widetilde\phi$ commutes with the action of $\ZZ^2$. The properties of the fundamental domain $\Delta_n$ and the definition of the constant $d_3$ imply that  $\phi_t(x_n)\in T_n$ and $\mathrm{dist}\left(\phi_t(x_n),\partial T_n\right) > D+d$. This completes the proof of claim~\ref{claim-4}. See figure~\ref{f.triangle}.
\end{proof}

Claim~\ref{claim-3} and~\ref{claim-4} imply that, for $n$ large enough, for every $t\in [0,p_n]$, one has 
$$\left\|\widetilde\phi_t(x_n) - \widetilde\phi_t(y_n)\right\| \geq \left\|\widetilde\phi_t(x_n) - v_n(t)\right\| - \left\|v_n(t) - \widetilde\phi_t(y_n)\right\| \geq d.$$
Now consider the curves $\alpha_n:\RR\to\RR^2$ and $\beta_n:\RR\to\RR^2$, defined by $\alpha_n(t)=\widetilde\phi_t(x_n)$ and $\beta_n(t)=\widetilde\phi_t(y_n)$. The curve $\tau_n|{[0,p_n]}$ is a parametrization of the boundary of the triangle $T_n$~; the image of the curve $\alpha_n$ is contained in $T_n$. Therefore we have 
$$\mathrm{Linking}(\alpha_n,\tau_n)=\pm\frac{1}{p_n}$$
(the sign depends on the orientation induced by the parametrization $\tau_n$ on $\partial T_n$). Claim~3 implies that $\|\tau_n(t)-\beta_n(t)\|\leq D$ for every $t\in\RR$ and every $n\in\NN\setminus\{0\}$. Claim~4 implies that there exists an integer $n_0$ such that $\|\alpha_n(t)-\tau_n(t)\|>D+d$ for every $t\in\RR$ and every $n\geq n_0$. According to lemma~\ref{lem3}, this implies that $\mathrm{Linking}(\alpha_n,\beta_n)=\mathrm{Linking}(\alpha_n,\tau_n)$ for every $n\geq n_0$. Lastly, observe that $\mathrm{Linking}(\alpha_n,\beta_n)=\mathrm{Linking}\left(\widetilde J, x_n , y_n\right)$ (by definition of the curves $\alpha_n$ and $\beta_n$). So we finally get
$$\mathrm{Linking}\left(\widetilde J, x_n , y_n\right)=\pm \frac{1}{p_n}$$
for every $n\geq n_0$. This completes the proof of lemma~\ref{lem6}. See figure~\ref{f.triangle}.
\end{proof}

\begin{figure}[h]
\label{f.triangle}
\ifx\JPicScale\undefined\def\JPicScale{1}\fi
\psset{unit=\JPicScale mm}
\psset{linewidth=0.3,dotsep=1,hatchwidth=0.3,hatchsep=1.5,shadowsize=1,dimen=middle}
\psset{dotsize=0.7 2.5,dotscale=1 1,fillcolor=black}
\psset{arrowsize=1 2,arrowlength=1,arrowinset=0.25,tbarsize=0.7 5,bracketlength=0.15,rbracketlength=0.15}
\begin{pspicture}(0,0)(136.79,89.82)
\pspolygon[linewidth=0.1,fillcolor=lightgray,fillstyle=solid](63.27,49.45)(71,49.45)(71,41.34)(63.27,41.34)
\pspolygon[linewidth=0.05,fillcolor=lightgray,fillstyle=solid](29.81,29.39)(34.83,29.39)(34.83,25.74)(29.81,25.74)
\pspolygon[linewidth=0.05,fillcolor=lightgray,fillstyle=solid](53.02,27.77)(58.04,27.77)(58.04,24.12)(53.02,24.12)
\pspolygon[linewidth=0.05,fillcolor=lightgray,fillstyle=solid](76.22,26.15)(81.25,26.15)(81.25,22.5)(76.22,22.5)
\pspolygon[linewidth=0.05,fillcolor=lightgray,fillstyle=solid](99.43,24.53)(104.46,24.53)(104.46,20.88)(99.43,20.88)
\pspolygon[linewidth=0.05,fillcolor=lightgray,fillstyle=solid](111.04,35.87)(116.07,35.87)(116.07,32.22)(111.04,32.22)
\pspolygon[linewidth=0.05,fillcolor=lightgray,fillstyle=solid](99.43,48.84)(104.46,48.84)(104.46,45.19)(99.43,45.19)
\pspolygon[linewidth=0.05,fillcolor=lightgray,fillstyle=solid](87.83,61.8)(92.86,61.8)(92.86,58.16)(87.83,58.16)
\pspolygon[linewidth=0.05,fillcolor=lightgray,fillstyle=solid](76.22,74.77)(81.25,74.77)(81.25,71.12)(76.22,71.12)
\pspolygon[linewidth=0.05,fillcolor=lightgray,fillstyle=solid](53.02,76.39)(58.04,76.39)(58.04,72.75)(53.02,72.75)
\pspolygon[linewidth=0.05,fillcolor=lightgray,fillstyle=solid](41.41,65.05)(46.44,65.05)(46.44,61.4)(41.41,61.4)
\pspolygon[linewidth=0.05,fillcolor=lightgray,fillstyle=solid](122.64,22.91)(127.67,22.91)(127.67,19.26)(122.64,19.26)
\pspolygon[linewidth=0.05,fillcolor=lightgray,fillstyle=solid](18.2,42.35)(23.23,42.35)(23.23,38.71)(18.2,38.71)
\pspolygon[linewidth=0.05,fillcolor=lightgray,fillstyle=solid](29.81,53.7)(34.83,53.7)(34.83,50.05)(29.81,50.05)
\pspolygon[linewidth=0.05,fillcolor=lightgray,fillstyle=solid](64.62,87.74)(69.65,87.74)(69.65,84.09)(64.62,84.09)
\pspolygon[linewidth=0.05,fillcolor=lightgray,fillstyle=solid](6.6,31.01)(11.63,31.01)(11.63,27.36)(6.6,27.36)
\psline[fillcolor=lightgray,fillstyle=solid](9.11,29.19)(67.13,85.91)
\psline[fillcolor=lightgray,fillstyle=solid](67.13,85.91)(125.16,21.08)
\psline[fillcolor=lightgray,fillstyle=solid](125.16,21.08)(9.11,29.19)
\rput{0}(9.11,29.16){\psellipse[linestyle=none,fillstyle=solid](0,0)(0.45,0.45)}
\rput{0}(67.13,85.88){\psellipse[linestyle=none,fillstyle=solid](0,0)(0.44,0.44)}
\rput{90}(20.72,40.51){\psellipse[linestyle=none,fillstyle=solid](0,0)(0.45,0.45)}
\rput{0}(32.32,51.85){\psellipse[linestyle=none,fillstyle=solid](0,0)(0.45,0.45)}
\rput{0}(125.16,21.05){\psellipse[linestyle=none,fillstyle=solid](0,0)(0.44,0.44)}
\rput{0}(43.92,63.19){\psellipse[linestyle=none,fillstyle=solid](0,0)(0.44,0.44)}
\rput{0}(55.52,74.57){\psellipse[linestyle=none,fillstyle=solid](0,0)(0.47,0.47)}
\rput{0}(101.95,22.68){\psellipse[linestyle=none,fillstyle=solid](0,0)(0.45,0.45)}
\rput{0}(113.55,34.01){\psellipse[linestyle=none,fillstyle=solid](0,0)(0.44,0.44)}
\rput{0}(101.95,46.99){\psellipse[linestyle=none,fillstyle=solid](0,0)(0.45,0.45)}
\rput{90}(90.35,59.96){\psellipse[linestyle=none,fillstyle=solid](0,0)(0.45,0.45)}
\rput{0}(78.74,72.92){\psellipse[linestyle=none,fillstyle=solid](0,0)(0.45,0.45)}
\rput{0}(32.32,27.54){\psellipse[linestyle=none,fillstyle=solid](0,0)(0.45,0.45)}
\rput{0}(55.53,25.92){\psellipse[linestyle=none,fillstyle=solid](0,0)(0.45,0.45)}
\rput{0}(78.74,24.3){\psellipse[linestyle=none,fillstyle=solid](0,0)(0.45,0.45)}
\pscustom[linewidth=0.2,linecolor=red]{\psbezier(8.87,27.64)(6.86,28.59)(6.81,29.64)(8.72,31.12)
\psbezier(10.62,32.6)(12.24,33.48)(14.14,34.05)
\psbezier(16.04,34.62)(16.85,35.35)(16.85,36.48)
\psbezier(16.85,37.61)(17.2,38.46)(18.01,39.32)
\psbezier(18.82,40.17)(19.98,40.89)(21.88,41.75)
\psbezier(23.77,42.6)(24.93,43.69)(25.74,45.39)
\psbezier(26.55,47.1)(27.48,48.44)(28.84,49.86)
\psbezier(30.19,51.27)(31.24,52.37)(32.32,53.5)
\psbezier(33.4,54.63)(34.8,55.61)(36.96,56.74)
\psbezier(39.13,57.87)(40.87,59.09)(42.76,60.79)
\psbezier(44.66,62.49)(45.82,64.07)(46.63,66.06)
\psbezier(47.44,68.05)(48.02,69.26)(48.56,70.12)
\psbezier(49.11,70.97)(49.57,70.72)(50.12,69.3)
\psbezier(50.66,67.88)(50.43,67.51)(49.34,68.09)
\psbezier(48.26,68.66)(48.37,69.26)(49.72,70.12)
\psbezier(51.08,70.97)(51.89,71.69)(52.44,72.54)
\psbezier(52.98,73.4)(53.44,74)(53.98,74.57)
\psbezier(54.53,75.14)(55.22,75.62)(56.3,76.19)
\psbezier(57.39,76.76)(58.43,77.73)(59.78,79.43)
\psbezier(61.14,81.13)(62.3,82.47)(63.66,83.88)
\psbezier(65.01,85.3)(66.63,85.55)(69.07,84.7)
\psbezier(71.51,83.84)(72.55,82.75)(72.55,81.05)
\psbezier(72.55,79.35)(72.2,79.23)(71.39,80.65)
\psbezier(70.58,82.06)(70.69,82.43)(71.78,81.86)
\psbezier(72.86,81.29)(73.78,80.44)(74.87,79.02)
\psbezier(75.96,77.61)(76.54,76.27)(76.8,74.57)
\psbezier(77.07,72.87)(77.77,71.9)(79.12,71.33)
\psbezier(80.48,70.76)(81.87,69.91)(83.77,68.49)
\psbezier(85.67,67.07)(86.71,65.24)(87.25,62.41)
\psbezier(87.79,59.57)(89.41,57.27)(92.66,54.72)
\psbezier(95.92,52.16)(97.89,50.22)(99.24,48.23)
\psbezier(100.59,46.24)(101.75,44.78)(103.1,43.36)
\psbezier(104.46,41.95)(105.51,40.61)(106.59,38.91)
\psbezier(107.68,37.21)(107.91,37.57)(107.37,40.12)
\psbezier(106.82,42.68)(107.63,42.07)(110.07,38.1)
\psbezier(112.51,34.13)(114.6,31.34)(117.04,28.78)
\psbezier(119.47,26.22)(120.98,24.52)(122.06,23.1)
\psbezier(123.15,21.69)(123.61,20.84)(123.61,20.27)
\psbezier(123.61,19.7)(120.59,19.58)(113.56,19.87)
\psbezier(106.52,20.15)(102.92,20.76)(101.56,21.89)
\psbezier(100.21,23.02)(98.7,23.27)(96.54,22.7)
\psbezier(94.37,22.13)(92.16,22.25)(89.18,23.1)
\psbezier(86.21,23.96)(84.23,24.08)(82.61,23.51)
\psbezier(80.99,22.94)(79.24,23.06)(76.8,23.92)
\psbezier(74.37,24.77)(72.04,25.37)(69.06,25.94)
\psbezier(66.09,26.51)(63.65,26.75)(60.94,26.75)
\psbezier(58.24,26.75)(55.8,26.99)(52.82,27.56)
\psbezier(49.85,28.14)(46.71,28.38)(42.38,28.38)
\psbezier(38.05,28.38)(35.72,28.01)(34.64,27.16)
\psbezier(33.56,26.31)(31.7,26.43)(28.45,27.56)
\psbezier(25.2,28.7)(22.65,28.82)(19.94,27.97)
\psbezier(17.23,27.12)(15.71,26.66)(14.87,26.45)
\psbezier(14.03,26.24)(13.38,26.17)(12.7,26.22)
\psbezier(12.03,26.27)(10.88,26.7)(8.87,27.64)
\closepath}
\pscustom[linewidth=0.2,linecolor=blue]{\psbezier(69.07,49.45)(69.34,50.58)(69.69,51.07)(70.23,51.07)
\psbezier(70.77,51.07)(71.7,50.58)(73.32,49.45)
\psbezier(74.94,48.32)(75.06,47.34)(73.71,46.2)
\psbezier(72.36,45.07)(71.32,44.82)(70.23,45.39)
\psbezier(69.15,45.96)(69.26,46.44)(70.62,47.02)
\psbezier(71.97,47.59)(71.85,47.1)(70.23,45.39)
\psbezier(68.61,43.69)(67.91,43.32)(67.91,44.18)
\psbezier(67.91,45.03)(67.56,45.88)(66.75,47.02)
\psbezier(65.94,48.15)(66.05,48.52)(67.14,48.24)
\psbezier(68.22,47.95)(68.8,48.32)(69.07,49.45)
\closepath}
\rput{0}(66.36,47.8){\psellipse[linestyle=none,fillstyle=solid](0,0)(0.45,0.45)}
\psline[linewidth=0.1,fillcolor=cyan,fillstyle=solid]{<->}(74.59,48.65)(89.57,57.55)
\rput[bl]{30}(78.88,51.76){$>d$}
\rput(4.11,37.77){$\widetilde R$}
\psline[linewidth=0.1,fillcolor=cyan,fillstyle=solid]{<-}(17.95,42.79)(15.47,50.31)
\psline[linewidth=0.1,fillcolor=cyan,fillstyle=solid]{<-}(67.83,86.64)(74.32,89.68)
\rput[bl](74.87,88.81){$y+n\ell_1(p_1,q_1)$}
\rput[bl](29.84,89.53){$\widetilde R+n\ell_1(p_1,q_1)$}
\rput[bl](24.58,48.72){}
\rput[bl](21.68,48.29){}
\psline[linewidth=0.1,fillcolor=cyan,fillstyle=solid]{<-}(65.75,48.43)(63.13,53.79)
\psline[linewidth=0.1,fillcolor=cyan,fillstyle=solid]{<-}(64.1,42.5)(58.57,43.66)
\rput[bl](53.57,43.93){$\Delta_n$}
\rput[bl](61.61,54.51){$x_n$}
\psline[linewidth=0.1]{<-}(7.5,27.86)
(6.07,23.54)(10.08,21.08)
\rput[bl](10.71,19.2){$y_n$}
\psline[linewidth=0.1,fillcolor=cyan,fillstyle=solid]{<-}(6.79,31.07)(4.02,36.07)
\psline[linewidth=0.1]{<-}(126.25,21.25)
(136.79,19.64)(131.79,6.43)
\rput[l](98.75,3.93){$y+n(\ell_1(p_1,q_1)+\ell_2(p_2,q_2))$}
\psline[linewidth=0.1,fillcolor=cyan,fillstyle=solid]{<-}(64.37,87.8)(55.67,89.82)
\rput[bl](1.42,50.55){$\widetilde R+\ell_1(p_1,q_1)$}
\rput[l](65.58,12.59){$\widetilde R+n(\ell_1(p_1,q_1)+\ell_2(p_2,q_2))$}
\psline[linewidth=0.1]{<-}(127.09,18.91)
(128.75,14.57)(113.91,13.04)
\rput{0}(7.68,28.57){\psellipse[linestyle=none,fillstyle=solid](0,0)(0.45,0.45)}
\psline[linewidth=0.1]{<-}(10.09,29.46)
(16.61,32.41)(20.89,33.13)
\rput(22.95,33.48){$y$}
\rput(25.62,34.2){}
\end{pspicture}
\caption{This figure depicts the situation for $n$ bigger than $n_0$. The blue curve represents the trajectory $t\mapsto \widetilde\phi_t(x_n)$. The red curve represents the trajectory $t\mapsto \widetilde\phi_t(y_n)$. The triangle $T_n$ is drawn in black. The red curve stays at distance less than $D$ of the black triangle. The distance between the blue curve and the black triangle is bigger than $D+d$. Hence, the distance between the red curve and the blue curve is bigger than $d$.}
\end{figure}
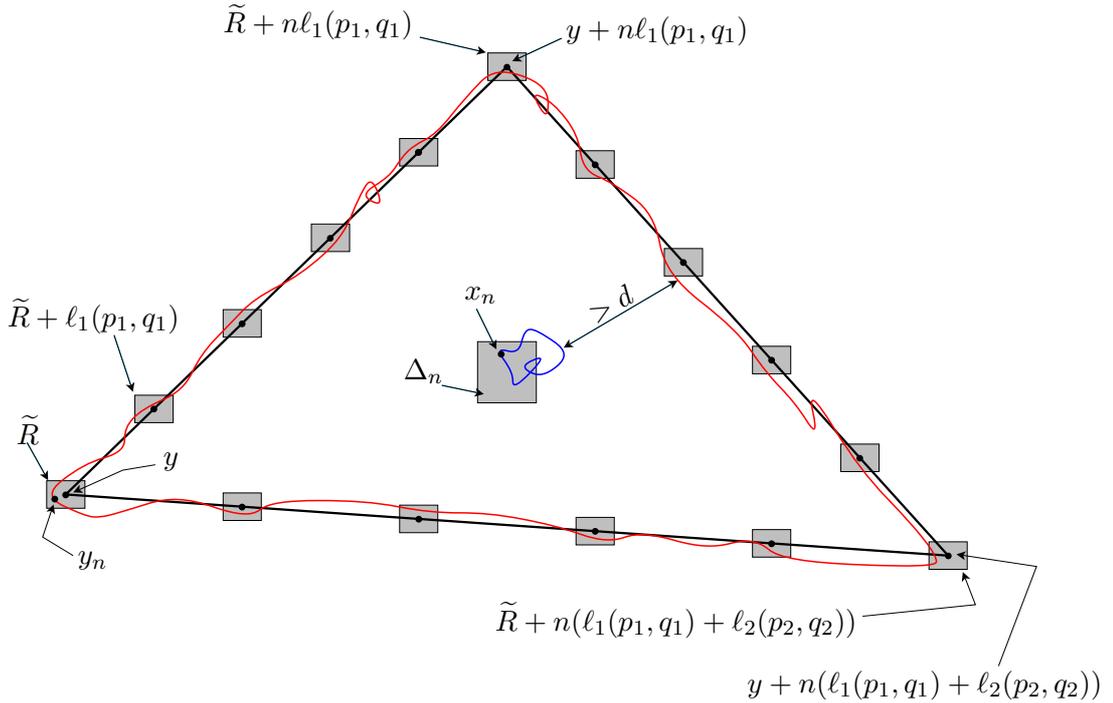

\bibliographystyle{plain}

\end{document}